\documentclass[10pt]{amsart}

\usepackage{amssymb, amscd, amsmath, amsthm, epsf, epsfig, latexsym, enumerate, amsxtra}

\usepackage{url,hyperref}

\usepackage{tikz}
\usepackage{tikz-cd}

\usepackage[all]{xy}

\usepackage{a4wide}

\theoremstyle{plain}
\newtheorem{theorem}{Theorem}
\newtheorem{lemma}[theorem]{Lemma}
\newtheorem{prop}[theorem]{Proposition}
\newtheorem{cor}[theorem]{Corollary}

\newtheorem*{thm}{Theorem}

\newtheorem*{maina}{Theorem A}
\newtheorem*{mainb}{Theorem B}
\newtheorem*{mainc}{Theorem C}
\newtheorem*{maind}{Theorem D}

\theoremstyle{definition}

\DeclareMathOperator{\aug}{aug}
\DeclareMathOperator{\coker}{coker}
\DeclareMathOperator{\ev}{ev}
\DeclareMathOperator{\im}{im}
\DeclareMathOperator{\Hom}{Hom}
\DeclareMathOperator{\vcd}{vcd} 

\begin{document}

\title{Poincar\'e duality complexes with highly connected universal cover}

\author{Beatrice Bleile, Imre Bokor and Jonathan A. Hillman}

\address{Mathematics, School of Science and Technology\\  University of New England, NSW 2351\\ Australia}
\email{bbleile@une.edu.au}

\address{13 Holmes Avenue\\  Armidale, NSW 2350\\ Australia}
\email{ibokor@bigpond.net.au}

\address{School of Mathematics and Statistics\\
     University of Sydney, NSW 2006\\
      Australia }

\email{jonathan.hillman@sydney.edu.au}

\begin{abstract}

Turaev conjectured that the classification, realization and splitting results for Poin\-car{\'{e}} duality complexes of dimension $3$ (\emph{$PD_{3}$-complexes}) generalize to $PD_{n}$-complexes with $(n-2)$-connected universal cover for $n \ge 3$. Baues and Bleile showed that such complexes are classified, up to oriented homotopy equivalence, by the triple consisting of their fundamental group, orientation class and the image of their fundamental class in the homology of the fundamental group, verifying Turaev's conjecture on classification. 

We prove Turaev's conjectures on realization and splitting. 
We show that a triple $(G, \omega, \mu)$, comprising a group, $G$, 
a cohomology class $\omega \in H^{1}\left(G; \mathbb{Z}/2\mathbb{Z}\right)$ 
and a homology class $\mu \in H_{n}(G; \mathbb{Z}^{\omega})$ can be realized by a $PD_{n}$-complex with $(n-2)$-connected universal cover if and only if the Turaev map applied to $\mu$ yields an equivalence. We then show that a $PD_{n}$-complex with $(n-2)$-connected universal cover is a non-trivial connected sum of two such complexes if and only if its fundamental group is a non-trivial free product of groups.

We then consider the indecomposable $PD_n$-complexes of this type. When $n$ is odd the results are similar to those for the case $n=3$.  The indecomposables are either aspherical or have virtually free fundamental group.  When $n$ is even the indecomposables include manifolds which are neither aspherical nor have virtually free fundamental group,  but if the group is virtually free and has no dihedral subgroup of order $>2$ then it has two ends.
\end{abstract}



\maketitle

\section{Introduction}\label{section:intro}

Hendricks classified Poincar{\'{e}} duality complexes of dimension $3$ (\emph{$PD_{3}$-complexes}) up to oriented homotopy equivalence using the ``fundamental triple''\!,  comprising  the fundamental group, the orientation character and the image of the fundamental class in the homology of the fundamental group.

Turaev \cite{Tu} gave an alternative proof of Hendricks' result and provided necessary and sufficient conditions for a triple $(G, \omega, \mu)$, comprising a group, $G$, a cohomology class, $\omega \in H^{1}\!\left(G ; \mathbb{Z}/2\mathbb{Z}\right)$, and a homology class, $\mu \in H_{3}(G ; \mathbb{Z}^{\omega})$, where $\mathbb{Z}^{\omega}$ denotes the integers, $\mathbb{Z}$, regarded as a right $\mathbb{Z}[G]$-module with respect to the \emph{twisted} structure induced by $\omega$, to be the fundamental triple of a $PD_{3}$-complex. Central to this was that the image of $\mu$ under a specific homomorphism, which we call \emph{the Turaev map}, be an isomorphism in the stable category of $\mathbb{Z}[G]$, that is to say, a homotopy equivalence of $\mathbb{Z}[G]$-modules. 

The results on classification and splitting allowed Turaev to show that a $PD_{3}$-complex is a non-trivial connected sum of two $PD_{3}$-complexes if and only if its fundamental group is a non-trivial free product of groups. He conjectured that these results hold for all $PD_{n}$-complexes with $(n-2)$-connected universal cover.

Baues and Bleile  classified Poincar{\'{e}} duality complexes 
of dimension 4 in \cite{HJB-BB2010}. 
Their analysis showed that a $PD_{n}$-complex $X$,  
with $n\ge 3$, is classified up to oriented homotopy equivalence 
by the triple comprising its $(n-2)$-type, $P_{n-2}(X)$, its orientation character, $\omega=\omega_X \in 
{H^{1}\left(\pi_{1}(X) ; \,  \mathbb{Z}/2\mathbb{Z}\right)}$, 
and its fundamental class, $[X]\in H_{n}(X; \mathbb{Z}^{\omega})$.
(We assume that all spaces have basepoints.
Thus every map $f:X\to{Y}$ has a preferred lift to a map of universal covers.
Hence if $f^*\omega_Y=\omega_X$ there is a well-defined homomorphism 
$H_n(f;\mathbb{Z}^\omega)$, and it is meaningful to
say that $f$ is ``oriented",
i.e., that $f_*[X]=[Y]$.
See  \cite{Ta} for a discussion of this issue.)

They called this the \emph{fundamental triple of $X$}, as it is a generalization of Hendricks' ``fundamental triple'', for the $(n-2)$-type of $X$ determines its fundamental group. Moreover, when the universal cover of the complex is $(n-2)$-connected --- automatically the case when $n=3$  ---  the $(n-2)$-type is an Eilenberg-Mac Lane space of type $(\pi_{1}(X),1)$, so that the $(n-2)$-type  and the fundamental group determine each other completely, reducing their fundamental triple to that of Hendricks.

Turaev's conjecture on classification is a direct consequence:

\begin{thm}
There is an oriented homotopy equivalence between two $PD_{n}$-complexes with $(n-2)$-connected universal cover if and only if their fundamental triples are isomorphic.
\end{thm}

We prove Turaev's conjectures on realization and splitting. These are, respectively, Theorem A and Theorem B below.

\medskip

Recall that the group $G$ is of type FP$_{n}$ if and only if the trivial $\mathbb{Z}[G]$-module $\mathbb{Z}$ has a projective resolution, $\mathbf{P}$, with $P_{j}$ finitely generated for $j \le n$.

\begin{maina}
  Let $G$ be a finitely presentable group, $\omega$ a cohomology class in $H^{1}\left(G; \mathbb{Z}/2\mathbb{Z}\right)$, and $\mu$ a homology class in $H_{n}(G; \mathbb{Z}^{\omega})$, with $n \ge 3$.
  
 If $G$ is of type FP$_{n-1}$, with $H^{i}(G; \,^{\omega}\! \mathbb{Z}[G]) =0$ for $1 < i \le n-1$, then $(G, \omega, \mu)$ can be realized as the fundamental triple of a $PD_{n}$-complex with $(n-2)$-connected universal cover if and only if the Turaev map applied to $\mu$ yields an isomorphism in the stable category of $\mathbb{Z}[G]$.
\end{maina}

\begin{mainb}
A $PD_{n}$-complex with $(n-2)$-connected universal cover decomposes as a non-trivial connected sum of two such $PD_{n}$-complexes if only if its fundamental group decomposes as a non-trivial free product of groups.
\end{mainb}

Thus, it is enough to investigate $PD_{n}$-complexes  with $(n-2)$-connected universal cover whose fundamental group is indecomposable as free product, 
and we turn to the analysis of such complexes.
Our arguments here exploit the interaction of Poincar\'e duality with the \emph{Chiswell sequence associated with a graph of groups}  (cf.~\cite{Cr} and \cite{Hi12}).

The parity of the dimension $n$ is significant.

When $n$ is odd, indecomposable orientable $PD_n$-complexes are either aspherical or have virtually free fundamental groups,  and the arguments of \cite{Hi12} give similar constraints on the latter class of groups. (See Section \ref{section:crisp}.)
However, implementing the Realization Theorem may be difficult,
and we do not consider this case further.

When $n$ is even there are indecomposable fundamental groups, $\pi$, 
with virtual cohomological dimension $n$ --- $\vcd\pi=n$ --- and infinitely many ends.
Our strongest results are for groups which are indecomposable and virtually free.

\begin{mainc}
Let $X$ be a $PD_{2k}$-complex with $(2k-2)$-connected universal cover. If $\pi=\pi_1(X)$ is virtually free, indecomposable as a free product and has no dihedral subgroup of order $>2$ then either $X\simeq\mathbb{R}P^{2k}$ or
$\pi$ has two ends and its maximal finite subgroups have cohomological period dividing $2k$. Hence $\widetilde{X}\simeq{S^{2k-1}}$. If, moreover, $X$ is orientable then  $\pi/\pi'\cong\mathbb{Z}$.
\end{mainc}
 
In particular, Theorem C applies to closed 4-manifolds $M$ with $\pi_2(M)=0$ and such fundamental groups.  There is no geometric connected sum decomposition theorem for 4-manifolds currently known that corresponds to Theorem B.

There is also a realization result,  when $\pi\cong{F}\rtimes_\theta\mathbb{Z}$ with $F$ finite (i.e., when $\pi/\pi'\cong\mathbb{Z}$).

\begin{maind}
If  $\pi\cong{F}\rtimes_\theta\mathbb{Z}$ where $F$ is finite then $\pi=\pi_1(X)$ for some $PD_{2k-1}$-complex $X$ with $\widetilde{X}\simeq{S^{2k-1}}$ if and only if  $F$ has cohomological period dividing $2k$ and $H_{n-1}(\theta;\mathbb{Z})=\pm1$.
\end{maind}

In particular, we have not determined the possibilities when $\pi$ has dihedral subgroups, and we do not know whether there are examples with $D_4=(\mathbb{Z}/2\mathbb{Z})^2$ as subgroup. As seems common in topology,  there appear to be difficulties associated with 2-torsion!

Section \ref{section:notation} summarizes background material and fixes notation.

Section \ref{section:necessity} contains the formulation and proof of the necessity of the condition for the realization of a  fundamental triple by a $PD_{n}$-complex with $(n-2)$-connected universal cover.

Section \ref{section:sufficiency} completes the proof of Theorem A, with the sufficiency of the condition in Section \ref{section:necessity}.

Section \ref{section:connsum} contains the proof of Theorem B, showing how the fundamental triple detects connected sums.

Section \ref{section:crisp} starte the discussion of indecomposable $PD_{n}$-complexes with $(n-2)$-connected universal covers, beginning  with the Chiswell sequence associated with a graph of groups and Crisp's centralizer condition. 

In Section \ref{section:otherconsequences} we give some supporting results,  and construct examples of indecomposable groups $\pi$ with infinitely many ends and $\vcd\pi=n$ which are the fundamental groups of closed $n$-manifolds with $(n-2)$-connected universal cover.

In Section \ref{section:virtuallyfree} we show that finite subgroups of $\pi$ of odd order are metacyclic.

In Sections \ref{section:nodihedralsub} and \ref{section:construct} we prove Theorems C and D.

Section \ref{section:periodic} concludes by briefly considering possible examples with $D_4$ as a subgroup.
{
}

\section{Background and Notation}\label{section:notation}

This section summarizes background material and fixes notation for the rest of the paper. Details and further references can be found in \cite{BB2010}, \cite{Cr}, \cite{Hi12}, \cite{DD} and \cite{Serre}.

 Let $\Lambda$ be the integral group ring, $\mathbb{Z}[G]$, of the finitely presentable group, $G$. We write $I$ for the \emph{augmentation ideal}, the kernel of the augmentation map 
 \[
 \aug \colon \Lambda \longrightarrow \mathbb{Z}, \quad \sum_{g \in G} n_{g} g \longmapsto \sum_{g \in G} n_{g}
 \]
  where $\mathbb{Z}$ is a $\Lambda$-bi-module with trivial $\Lambda$ action. Each cohomology class $\omega \in H^{1}\left(G; \mathbb{Z}/2\mathbb{Z}\right)$ may be viewed as a group homomorphism $\omega \colon G \longrightarrow  \mathbb{Z}/2\mathbb{Z} =\{ 0, 1\}$ and yields an anti-isomorphism
\begin{equation*}
   \overline{\phantom{B}} \colon \Lambda \longrightarrow \Lambda, \quad \lambda= \sum_{g \in G} n_{g} g \longmapsto \overline{\lambda} = \sum_{g \in G} (-1)^{\omega(g)}  n_{g} g^{-1}  
   \end{equation*}

Consequently, a right $\Lambda$-module, $A$, yields the \emph{conjugate left $\Lambda$-module}, $^{\omega}\!A$, with action given by 
\[
   \lambda  \mathbf{\centerdot} a := a .\overline{\lambda}
\]
  for $\lambda \in \Lambda, a \in A$. Similarly, a left $\Lambda$-module $B$ yields the conjugate right $\Lambda$-module, $B^{\omega}$.

Given left $\Lambda$-modules $A_{j}, B_{i}$ for $ 1 \le i \le k$ and $1 \le j \le \ell$, we sometimes write the $\Lambda$-morphism $\psi \colon \bigoplus_{j=1}^{\ell} A_{j} \longrightarrow \bigoplus_{i=1}^{k}B_{i}$
in matrix form as
\(
    \big[\psi_{ij}\big]_{k \times \ell}
\) 
for $\psi_{ij} = pr_{i} \circ \psi \circ in_{j} \colon A_{j} \longrightarrow B_{i}$, where $in_{j}$ is the $j^{\text{th}}$ natural inclusion and $pr_{i}$ the $i^{\text{th}}$ natural projection of the direct sum. The composition of such morphisms is given by matrix multiplication.

If $B$ is a left $\Lambda$-module and $M$  a $\Lambda$-bi-module, then $\Hom_{\Lambda}(B,M)$ is a right $\Lambda$-module with  action given by 
\[
\varphi.\lambda \colon B \longrightarrow M, \quad b \longmapsto \varphi(b).\lambda
\]
The dual of the left $\Lambda$-module $B$ is the left $\Lambda$-module $B^{\ast} = \,^{\omega}\!\Hom_{\Lambda}(B, \Lambda)$. The construction of the dual defines an endofunctor on the category of left $\Lambda$-modules. 

Evaluation  defines a natural transformation, $\varepsilon$, from the identity functor to the double dual functor, where for the left $\Lambda$-module $B$, 
\[
  \varepsilon_{B} \colon B \longrightarrow  B^{\ast\ast} = \, ^{\omega}\!\Hom_{\Lambda}\big(\,^{\omega}\!\Hom_{\Lambda}(B, \Lambda), \Lambda\big), \quad b \longmapsto \, \overline{\ev_{b}}
\]
with $\overline{\ev_{b}}$  defined by 
\[
  \overline{\ev_{b}} \colon \, ^{\omega}\!\Hom(B, \Lambda) \longrightarrow \Lambda, \quad \psi \longmapsto \overline{\psi(b)}
 \]
 The left $\Lambda$-module, $A$, defines the natural transfomation, 
\(
\eta
\),  from the functor \(A^{\omega} \otimes_{\Lambda} \underline{\phantom{B}} \) to the functor \( \Hom_{\Lambda}\big(\,^{\omega}\!\Hom_{\Lambda}(\,\underline{\phantom{B}} \, , \Lambda), A\big) \) where, for the left $\Lambda$-module $B$, 
\[
   \eta_{B} \colon A^{\omega} \otimes_{\Lambda} B \longrightarrow \Hom_{\Lambda}(B^{\ast}\! , A) = \Hom_{\Lambda}\big(\,^{\omega}\!\Hom_{\Lambda}(B, \Lambda), A\big)
\] 
is given by 
\[
\eta_{B}(a \otimes b ) \colon \psi \longmapsto \overline{\psi(b)}.a
\]
 for $a \otimes b \in A^{\omega}\otimes_{\Lambda} B$.  Both $\varepsilon$ and $\eta$ become natural equivalences when restricted to the category of finitely generated free $\Lambda$-modules.

The $\Lambda$-morphisms $f,g \colon A_{1} \longrightarrow A_{2}$ are \emph{homotopic} if and only if the $\Lambda$-morphism $f-g \colon A_{1} \longrightarrow A_{2}$ factors through a projective $\Lambda$-module $P$. Associated with $\Lambda$ is its \emph{stable category}, whose objects are all $\Lambda$-modules  and whose morphisms are all homotopy classes of $\Lambda$-morphisms. Thus,  an isomorphism in the stable category of $\Lambda$ is a homotopy equivalence of $\Lambda$-modules.

We work in the category of connected, well pointed $CW$-complexes and pointed maps.  We write $X^{[k]}$ for the $k$-skeleton of $X$, suppressing the base point from our notation. The inclusion of the $k$-skeleton into the $(k+1)$-skeleton induces homomorphisms $ \iota_{k, r} \colon \pi_{r}(X^{[k]}) \longrightarrow \pi_{r}(X^{[k+1]})$ and we write $\Gamma_{k+1}(X)$ for $\im(\iota_{k,k+1})$.

From now, we work with the fundamental group of $X$, $\pi = \pi_{1}(X)$, and its integral group ring $ \Lambda = \mathbb{Z}[\pi]$. We take $X$ to be a reduced $CW$-complex, so that $X^{[0]} = \{\ast\}$, and write $u \colon \widetilde{X}\longrightarrow X$ for the universal cover of $X$, fixing a base point for  $\widetilde{X}$ in $u^{-1}(\ast)$. We write  $\mathbf{C}(\widetilde{X})$ for  the cellular chain complex of $\widetilde{X}$ viewed as a complex of left $\Lambda$-modules. Since $X$ is reduced, $C_{0}(\widetilde{X}) = \Lambda$.

The homology and cohomology of $X$ we work with are the abelian groups
\begin{align*}
   {H}_{q}(X;A) &:= H_{q}(A \otimes_{\Lambda} \mathbf{C}(\widetilde{X}))\\
   {H}^{q}(X;B) &:= H^{q}\big(\Hom_{\Lambda}(\mathbf{C}(\widetilde{X}), B))
\end{align*}
where $A$ is a right $\Lambda$-module and $B$ is a left $\Lambda$-module. We write $H_{-q}(X;B)$ for $H^{q}(X;B)$, when treating cohomology as ``homology in negative degree'' simplifies the exposition.

An \emph{$n$-dimensional Poincar{\'e} duality complex} ($PD_{n}$-complex) comprises a reduced connected $CW$-complex, $X$,  whose  fundamental group, $\pi_{1}(X)$,  is finitely presentable, together with an \emph{orientation character}, 
$\omega=\omega_X \in H^{1}\left(\pi_{1} (X) ; \mathbb{Z}/2\mathbb{Z}\right)$, viewed as a group homomorphism 
\( 
 \pi_{1} (X) \longrightarrow  \mathbb{Z}/2\mathbb{Z}
 \), 
and a \emph{fundamental class }$[X] \in H_n(X; \mathbb Z^{\omega})$, such that  for every $r \in \mathbb{Z}$ and left $\mathbb{Z} [\pi_{1} (X)]$-module $M$, the \emph{cap product with $[X]$},
\[ 
\underline{\phantom{A}}\smallfrown [X] \colon  H^{r}(X; M) \longrightarrow H_{n-r}(X; M^{\omega}),\quad \alpha \longmapsto \alpha \smallfrown  [X]
\]
is an isomorphism of abelian groups. We denote this by $(X, \omega, [X])$.

Wall  (\cite{W1965}, \cite{W1967}) showed   that for $n >3$, every $PD_{n}$ complex is \emph{standard}, meaning that it is homotopically equivalent to an $n$-dimensional $CW$-complex with precisely one $n$-cell, whereas a $PD_{3}$ complex, $X$, is either standard, or \emph{weakly standard}, the latter meaning that it is homotopically equivalent to one of the form $X^{\prime} \cup e^{3}$, where $e^{3}$ is a 3-cell and $X^{\prime}$ is a 3-dimensional $CW$-complex with $H^{3}(X^{\prime};B) = 0$ for all coefficient modules $B$.

Baues introduced \emph{homotopy systems} to investigate when chain complexes and chain maps of free $\Lambda$-modules are realized by $CW$-complexes (\cite{HJB-1991}).

  Take an integer $n >1$. A \emph{homotopy system of order $(n+1)$} comprises 
  \begin{itemize}
  
  \item[(a)] a reduced $n$-dimensional $CW$-complex, $X$,
  
  \item[(b)]  a chain complex $\mathbf{C}$ of free $\Lambda$-modules coinciding with $\mathbf{C}(\widetilde{X}) $ in degree $q$ for $q \le n$,

  \item[(c)] a homomorphism, \( f_{n+1} \colon C_{n+1} \longrightarrow \pi_{n}(X) \) with  \( f_{n+1} \circ d_{n+2} = 0 \) rendering commutative the diagram
\[
     \begin{tikzpicture}[scale=0.8]
               \path node (Cn+1) at (-1.5,1) {$C_{n+1}$} 
                 node (pin) at (1.5,1)  {$\pi_{n}(X)$}
                 node (Cn) at (-1.5,-1) {$C_{n}$}
                 node (pirel) at (1.5,-1) {$\pi_{n}\big(X,X^{[n-1]}\big)$};
        \draw[->] (Cn+1)--(pin) node[midway,above] {$f_{n+1}$};
        \draw[->] (Cn+1) -- (Cn) node[midway,left] {$d_{n+1}$};
        \draw[->] (pin) -- (pirel) node[midway,right] {$j$};
        \draw[->] (pirel)--(Cn) node[midway,below] {$h_{n}$};
     \end{tikzpicture}
  \]
  where $j$ is induced by the inclusion $(X, \ast) \longrightarrow (X, X^{[n-1]})$, and
  \[
     h_{n} \colon \pi_{n}\big(X, X^{[n-1]}\big) \xrightarrow[\cong]{ \  \ u_{\ast}^{-1} \ \ } \pi_{n}\big(\widetilde{X}, \widetilde{X}^{[n-1]}\big) \xrightarrow[\cong]{ \  \ h \ \ } H_{n}\big(\widetilde{X}, \widetilde{X}^{[n-1]}\big) 
  \]
  is  the Hurewicz isomorphism, $h$, composed with $u^{-1}_{\ast}$, the inverse of the isomorphism induced by the universal covering map.

  \end{itemize}

\section{Formulation and Necessity of the Realization Conditions}\label{section:necessity}

For our generalization of Tuarev's realization condition to $PD_{n}$-complexes with $n \ge 3$, we introduce a set of functors, $\{G_{q} \mid q \in \mathbb{Z}\}$, from the category of chain complexes of projective left $\Lambda$-modules and homotopy classes of chain maps to the stable category of $\Lambda$. 

Given $\mathbf{f} \colon \mathbf{C} \longrightarrow \mathbf{D}$, a map of chain complexes of projective left $\Lambda$-modules, $\mathbf{C}$ and $\mathbf{D}$, put
\(
   G_{q}(\mathbf{C}) : = \coker \left(d_{q+1}^{\mathbf{C}}\colon C_{q+1} \to C_{q}\right) = C_{q}/\im \left(d_{q+1}^{\mathbf{C}}\right) 
\)  
and let $G_{q}(\mathbf{f})$ be the induced map of cokernels
\begin{equation*}
  \begin{tikzpicture}[scale=0.8]
     \draw 
         node (imc) at (-3,1) {$\im\big(d^{\mathbf{C}}_{q+1}\big)$}
         node (c) at (0,1) {$C_{q}$}
         node (gc) at (3,1) {$G_{q}(\mathbf{C})$}
         node (imd) at (-3,-1) {$\im\big(d^{\mathbf{D}}_{q+1}\big)$}
         node (d) at (0,-1) {$D_{q}$}
         node (gd) at (3,-1) {$G_{q}(\mathbf{D})$}
     ;
     \draw[>->] (imc)-- (c);
     \draw[->>] (c)--(gc);
     \draw[>->] (imd)-- (d);
     \draw[->>] (d)--(gd);
     \draw[->] (imc) -- (imd) ;
     \draw[->] (c) -- (d) node [midway, left] {$f_{q}$};
     \draw[->,dotted] (gc) -- (gd) node[midway, left] {$\exists !$} node [midway,right] {$G_{q}(\mathbf{f})$};
  \end{tikzpicture}
\end{equation*}

Direct verification shows that each $G_{q}$ is a functor from the category of chain complexes of left $\Lambda$-modules to the category of left $\Lambda$-modules. By Lemma 4.2 in \cite{BB2010}, chain homotopic maps $\mathbf{f} \simeq \mathbf{g} \colon \mathbf{C} \longrightarrow \mathbf{D}$ induce homotopic maps $G_{n}(\mathbf{f}) \simeq G_{n}(\mathbf{g})$, that is, $G_{q}(\mathbf{f}) - G_{q}(\mathbf{g})$ factors through a projective $\Lambda$-module. Hence, for each $q \in \mathbb{Z}$, $G_{q}$ is a functor from the category of chain complexes of projective left $\Lambda$-modules to the stable category of $\Lambda$.

Let $X$ be a $PD_{n}$-complex with $n \ge 3$, and let $\Lambda = \mathbb{Z}[\pi_{1}(X)]$.
By Remark 2.3 and Lemma 3.6 in \cite{HJB-BB2010}, we may assume that $X = X^{\prime} \cup e^{n}$ is standard (or weakly standard if $n=3$) with
\begin{equation*}
   C_{n}(\widetilde{X}) = C_{n}(\widetilde{X^{\prime}}) \oplus \Lambda e
\end{equation*} 
where $e$  corresponds to $e^{n}$, $1 \otimes e \in \mathbb{Z}^{\omega} \otimes_{\Lambda}C_{n}(\widetilde{X})$ is a cycle representing $[X]$, the fundamental class  of $X$,  and $e$ is a generator of $C_{n}(\widetilde{X})$.

Writing $F^{q}$ for $G_{q}\big(^{\omega}\!\Hom_{\Lambda}(\,\underline{\phantom{A}}\, , \Lambda)\big),$ Poincar{\'{e}} duality, together with Lemma 4.3 in \cite{BB2010}, provides the homotopy equivalence of $\Lambda$-modules
\begin{equation*}
  G_{-n+1}\big(\,\underline{\phantom{A}} \smallfrown (1 \otimes e)\big) \colon F^{n-1}\big(\mathbf{C}(\widetilde{X}) \big) \longrightarrow G_{1}\big(\mathbf{C}(\widetilde{X})\big)
\end{equation*}

Construct the $(n-2)$-type of $X$, $P = P_{n-2}(X)$, by attaching to $X$  cells of dimension $n$ and higher. Then the Postnikov section $p \colon X \rightarrow P$ is the identity on the $(n-1)$-skeleta and, for $0 \le i < n$, $C_{i}(\widetilde{X}) = C_{i}(\widetilde{P})$. Composing with the isomorphism $ \theta \colon G_{1}\big(\mathbf{C}(\widetilde{X})\big) \longrightarrow I, \ [c] \longmapsto d_{1}(c)$, we obtain the homotopy equivalence of left $\Lambda$-modules
\begin{equation}\label{homequi}
   \theta \circ G_{-n+1}\big(\,\underline{\phantom{A}} \smallfrown (1 \otimes e)\big) \colon F^{n-1}\big(\mathbf{C}(P)\big) \longrightarrow I
\end{equation}

We next construct the \emph{Turaev map}, which sends the image of the fundamental class of $X$ in the homology of the Postnikov section to the homotopy class of the homotopy equivalence (\ref{homequi}).

Let $\mathbf{C}$ be a chain complex of free left $\Lambda$-modules. This gives rise to the short exact sequence of chain complexes $0 \to \overline{I}\mathbf{C} \to \mathbf{C} \to \mathbb{Z}^{\omega}\otimes_{\Lambda} \mathbf{C} \to 0$, with associated connecting homomorphism $\delta_{r} \colon H_{r}\big(\mathbb{Z}^{\omega}\otimes_{\Lambda} \mathbf{C}\big) \to H_{r-1}(\overline{I}\mathbf{C}) $,

It is straightforward to verify that 
\begin{equation*}
   \widehat{\nu}_{\mathbf{C},r} \colon H_{r}(\overline{I}\mathbf{C}) \longrightarrow [F^{r}(\mathbf{C}), I] ,  \quad [\lambda.c] \longmapsto \Big[F^{r}\mathbf{C} \rightarrow I, \ [\varphi] \mapsto \overline{\varphi(\lambda.c)}\Big] 
\end{equation*}
is a homomorphism of groups. Composing $\widehat{\nu}_{\mathbf{C},r-1}$  with $\delta_{r}$ yields the Turaev map \begin{equation*}
    \nu_{\mathbf{C},r} \colon H_{r} \big(\mathbb{Z}^{\omega} \otimes_{\Lambda} \mathbf{C}\big) \longrightarrow \big[F^{r-1}(\mathbf{C}), I \big]
\end{equation*}

\begin{lemma}\label{lem:nu}
$\nu_{\mathbf{C}(\widetilde{P}), n}\big( p_{\ast}([X])\big) = \big[\theta \circ G_{-n+1}\big(\,\underline{\phantom{A}} \smallfrown (1\otimes e) \big)\big]$
\end{lemma}

\begin{proof}
  Take a diagonal $\Delta \colon \mathbf{C}(\widetilde{X}) \longrightarrow  \mathbf{C}(\widetilde{X}) \otimes_{\mathbb{Z}}  \mathbf{C}(\widetilde{X})$ and a chain homotopy $\alpha \colon  \mathbf{C}(\widetilde{X}) \longrightarrow  \mathbf{C}(\widetilde{X}) $ such that $id - (id \otimes \aug)\Delta = d \alpha + \alpha d$, where we have identified $\mathbf{C} \otimes_{\mathbb{Z}} \mathbb{Z}$ with $ \mathbf{C}$. 
  Let 
  \begin{equation*}
      \Delta e = e \otimes \lambda + \sum_{\ell}\sum_{0 \le   i   < n}x_{\ell,i} \otimes y_{\ell, n-i}
  \end{equation*}  
Direct calculation shows that $e = \aug(\lambda) e + \alpha de$. Since  $[1 \otimes e]$ generates ${H}_{n}\left(X; \mathbb{Z}^{\omega}\right) \cong \mathbb{Z}$, this yields $1 \otimes e = \aug(\lambda) \otimes e$, whence  $\aug( \lambda - 1) =0$.  Hence, $\lambda -1 \in I = \im(d_{1})$, or $ \lambda = 1 + d_{1}(c_{1})$ for some $c_{1} \in C_{1}(\widetilde{X})$.  Thus, given $\varphi \in \Hom_{\Lambda}\big(C_{n}(\widetilde{X}), \Lambda\big)$, 
 \begin{equation*}
 \varphi \smallfrown (1 \otimes e) = \overline{\varphi(e)}\big(1 + d_{1}(c_{1})\big)
 \end{equation*}
By direct calculation
 \begin{equation*}
  \Big(\theta \circ G_{-n+1}\big(\,\underline{\phantom{A}} \smallfrown (1 \otimes e)\big)\Big)\big([\varphi]\big) 
  = \overline{\varphi\big(d_{n}(e)\big)}(1 +d_{1}(c_{1}))
 \end{equation*}  
and
 \begin{equation*}
  \nu_{\mathbf{C}(\widetilde{P}), n}\big(p_{\ast}([X])\big)\big([\varphi]\big) 
  = \widehat{\nu}_{\mathbf{C}(\widetilde{P}), n}\big( [d_{n}(e)] \big)\big([\varphi]\big) 
  \end{equation*}  
Hence, by definition, $\nu_{\mathbf{C}(\widetilde{P}), n}\big(p_{\ast}([X])\big)$ is represented by the $\Lambda$-morphism
  \begin{equation*}
  F^{n-1}\big(\mathbf{C}(\widetilde{P})\big) \longrightarrow I, \quad [\varphi] \longmapsto \overline{\varphi\big(d_{n}(e)\big)}
  \end{equation*}
  To conclude the proof, note that $F^{n-1}\big(\mathbf{C}(\widetilde{P})\big) \longrightarrow I, \quad [\varphi] \longmapsto \overline{\varphi\big(d_{n}(e)\big)}.d_{1}(c_{1})$ is null-homotopic.
\end{proof}

As $\theta \circ G_{-n+1} \big(\,\underline{\phantom{A}} \smallfrown (1  \otimes e)\big)$ is a homotopy equivalence of $\Lambda$-modules, Lemma \ref{lem:nu} provides a necessary condition for realization.

\begin{theorem}\label{thm:necforrealisation}
Let $P$ be an $(n-2)$-type. Take $\omega \in H^{1}(P; \mathbb{Z}/2\mathbb{Z})$ and  $\mu \in H_{n}(P; \mathbb{Z}^{\omega})$. Then $(P, \omega, \mu)$ is the fundamental triple of a $PD_{n}$-complex only if  $\nu_{\mathbf{C}(\tilde{P}),n} (\mu)$ is a homotopy equivalence of left $\Lambda = \mathbb{Z}[\pi_{1}(P)]$-modules.
\end{theorem}
\begin{proof}
   Let $P$ be an $(n-2)$-type. Take $\omega \in H^{1}\big(P ;\mathbb{Z}/2\mathbb{Z}\big)$ and $\mu \in H_{n}(P; \mathbb{Z}^{\omega})$. Suppose that $(P, \omega, \mu)$ is the fundamental triple of the $PD_{n}$-complex, $X$. If $P^{\prime}$ is an $(n-2)$-type obtained by attaching to $X$ cells of dimension $n$ and higher, then there is  a homotopy equivalence $f \colon P \longrightarrow P^{\prime}$ with $f_{\ast}(\mu) = i_{\ast}([X])$, where $i \colon X \longrightarrow P^{\prime}$ is the inclusion. By Lemma \ref{lem:nu},
 \begin{equation*}
    \nu_{\mathbf{C}(\widetilde{P^{\prime}}) ,n}\big(i_{\ast}[X]\big) =\big[\theta \circ G_{-n+1}\big(\,\underline{\phantom{A}} \smallfrown (1  \otimes e)\big)\big]
\end{equation*}
and hence $\nu_{ \mathbf{C}(\widetilde{P}), n}(\mu)$ are homotopy equivalences of $\Lambda$-modules. 
\end{proof}

Now take a $PD_{n}$-complex, $X$, with $(n-2)$-connected universal cover. Then the $(n-2)$-type of $X$ is an Eilenberg-Mac Lane space, $K\big(\pi_{1}(X), 1\big)$, and we may identify the fundamental triple of $X$ with $\big( \pi_{1}(X), \omega, \mu\big)$, where $\mu$ is the image of $[X]$ in the group homology of $\pi_{1}(X)$.

\begin{lemma}\label{lem:Hi=0}
Let $(X, \omega, [X])$ be a $PD_{n}$-complex with $(n-2)$-connected universal cover. 
Then $\pi_1(X)$ is FP$_{n-1}$, and 
\( 
    H^{i}\big(\pi_{1}(X) ; \,^{\omega}\!\Lambda \big) = 0
\) 
for all $1 < i \leq n-1$.\end{lemma}

\begin{proof}
Since $X$ is a $PD_{n}$-complex, it is finitely dominated, and so is homotopy equivalent to a complex with finite $(n-1)$-skeleton.
Thus we may assume that $X^{[n-1]}$ is finite.
We  construct an Eilenberg-Mac Lane space $K = K\big(\pi_{1}(X), 1\big)$ from $X$ by attaching cells of dimension $n$ and higher. 
As the universal cover $\widetilde{X}$ of $X$ is $(n-2)$-connected, the cellular chain complexes of the universal covers $\widetilde{X}$ and $\widetilde{K}$ coincide in degrees below $n$, that is $C_{i}(\widetilde{X}) = C_{i}(\widetilde{K})$ for $0 \le i < n$. 
In particular, these modules are finitely generated, and so $\pi_1(X)$ is FP$_{n-1}$.
 
Moreover, for $ 1 < i \leq n-1$,
\[
  H^{i}\big(\pi_{1}(X) ; \,^{\omega}\!\Lambda \big) = H^{i}(\widetilde{X}; \,^{\omega}\!\Lambda)
\cong  H_{n-i}(\widetilde{X} ; \Lambda) = 0
\]
\end{proof}

Necessary conditions for realization are a corollary to Lemma \ref{lem:Hi=0} and Theorem \ref{thm:necforrealisation}.

\begin{cor}[\textbf{Conditions for Realizability}]\label{cor:ftem}
Let $G$ be a group. Take $\omega \in H^{1}\big(G; \mathbb{Z}/2 \mathbb{Z}\big)$ and $\mu \in H_{n}(G; \mathbb{Z}^{\omega})$. If  $(G, \omega, \mu)$ is the fundamental triple of a $PD_{n}$-complex with $(n-2)$-connected universal cover, then $G$ is a finitely presentable group of type FP$_{n-1}$, $H^{i}(G; \,^{\omega}\!\Lambda ) = 0$ for $1< i \le n -1$ and
$\nu_{\mathbf{C}(\widetilde{K}), n}(\mu)$ is a homotopy equivalence of $\Lambda = \mathbb{Z}[G]$ modules.
\end{cor}

\section{Sufficiency of the Realization Condition}\label{section:sufficiency}

We now establish the sufficiency of the realization conditions in Corollary \ref{cor:ftem}.

Let $G$ be a finitely presentable group of type FP$_{n-1}$, with $n \ge 3$ and $H^{i}(G; \,^{\omega}\! \mathbb{Z}[G]) =0$ for $1 < i \leq n-1$.

Take an Eilenberg-Mac Lane space, $K^{\prime} = K(G,1)$, and identify the (co-)homologies of $G$ and $K^{\prime}$. Given $\omega \in H^{1}\big(G; \mathbb{Z}/2\mathbb{Z} \big)$ and $\mu \in H_{n}\big( G; \mathbb{Z}^{\omega}\big) $, with $\nu_{\mathbf{C}(\widetilde{K^{\prime}}), n}(\mu) $ a class of homotopy equivalences of $\mathbb{Z}[G]$-modules, we construct a $PD_{n}$-complex, $X$, with $(n-2)$-connected universal cover and fundamental triple $(G, \omega, \mu)$. 

Let $\widetilde{K^{\prime}} \longrightarrow K^{\prime}$ be the universal covering of $K^{\prime}$. By the hypotheses on $G$, we can choose ${K^{\prime}}$ with finitely many cells in each dimension $<n$. 

Let $h \colon F^{n-1}\big( \mathbf{C}(\widetilde{K^{\prime}})\big) \longrightarrow I$ be a representative of  $\nu_{\mathbf{C}(\widetilde{K^{\prime}}), n}(\mu)$. Then $h$ is a homotopy equivalence of $\Lambda$-modules. By Theorem 4.1 and Observation 1 in \cite{BB2010}, $h$ factors as
 \begin{equation*}
  \begin{tikzcd}
     F^{n-1}\big(\mathbf{C}\big(\widetilde{K^{\prime}}\big)\big) \arrow[>->] {r} &
     F^{n-1}\big(\mathbf{C}\big(\widetilde{K^{\prime}}\big)\big) \oplus \Lambda^{m} \arrow[>->] {r} & 
     I \oplus P \arrow[->>] {r} & I
  \end{tikzcd}
 \end{equation*}
for some projective $\Lambda$-module, $P$, and $m \in \mathbb{N}$. Let $B = (e^{0} \cup e^{n-1}) \cup e^{n}$ be the $n$-dimensional ball and replace $K^{\prime}$ by the Eilenberg-Mac Lane space $K = K^{\prime} \vee \Big({\bigvee_{i=1}^{m} B} \Big)$.

Then $F^{n-1}\big(\mathbf{C}(\widetilde{K}) \big) = F^{n-1}\big(\mathbf{C}(\widetilde{K^{\prime}})\big) \oplus \Lambda^{m}$ and the factorization of $h$ becomes
\begin{equation*}
 \begin{tikzcd}
   h \colon F^{n-1}\big(\mathbf{C}(\widetilde{K})\big) \arrow[>->]{r}{j} & I \oplus P \arrow[->>]{r}{pr_{I}} & I
 \end{tikzcd}
\end{equation*}
Consider the $\Lambda$-morphism, $\varphi$, given by the composition 
\begin{equation*}
 \begin{tikzpicture}
   \draw
    node (1) at (0,0) {$C^{n-1}(\widetilde{K}) = \,  ^{\omega}\!\Hom_{\Lambda}\big(C_{n-1} (\widetilde{K}), \Lambda \big) $}  
    node (2) at (5,0) {$F^{n-1}\big(\mathbf{C}(\widetilde{K})\big)$}
    node (3) at (7.5,0) {$I \oplus P $}
    node (4) at (10,0) {$\Lambda \oplus P$}
   ;
 \draw[->>] (1)--(2) node [midway,above]{$p$};
 \draw[>->] (2)--(3) node [midway,above]{$j$};
 \draw[->] (3)--(4) node [midway, above] {\phantom{i}$\begin{bmatrix} i & 0 \\ 0 & id \end{bmatrix} $\phantom{a}};
 \end{tikzpicture}
 \end{equation*}
where $p$ is the projection onto the cokernel, and $ i \colon I \rightarrowtail \Lambda$  the inclusion. Since $F^{n-1}\big(\mathbf{C}(\widetilde{K}) \big) = C_{n-1}(\widetilde{K})/\im(d^{\ast}_{n-1})$ by definition, 
$\varphi \circ d^{\ast}_{n-1}=0$. As $C_{n-1}(\widetilde{K}^{[n-1]})$ is a finitely generated free $\Lambda$-module, the natural   map 
\[
^{\omega}\varepsilon \colon C_{n-1}(\widetilde{K}^{[n-1]}) \longrightarrow C_{n-1}(\widetilde{K}^{[n-1]})^{\ast\ast}
\]
is an isomorphism. Define
  \begin{equation*}
       d_{n} : = \, \big(^{\omega}\varepsilon\big)^{-1} \circ \,\varphi^{\ast}\colon (\Lambda \oplus P)^{\ast} \longrightarrow C_{n-1}(\widetilde{K})
    \end{equation*}
It follows from the naturality of \,$^{\omega}\varepsilon$ that $d_{n-1} \circ d_{n} = 0$.
 
\medskip
 
We first consider the case when $P$ is free, so that $P \cong \Lambda^{q}$ for some $q \in \mathbb{N}$ and $\Lambda \oplus P \cong \Lambda^{q+1}$.
 
 \medskip

Since  ${\widetilde{K}}^{[n-1]}$ is $(n-2)$-connected, the Hurewicz homomomorphism 
\[
h_{q} \colon \pi_{q}\big(\widetilde{K}^{[n-1]}\big) \longrightarrow H_{q}\big(\widetilde{K}^{[n-1]}\big)
\] 
is an isomorphism for $ q \le n-1$ and we obtain the map
\begin{align*}
      \varphi^{\prime}\colon \Lambda^{q+1} \cong (\Lambda \oplus P)^{\ast}  &\longrightarrow \ker(d_{n-1}) = H_{n-1}(\widetilde{K}^{[n-1]}) \xrightarrow[]{\ h_{n-1}^{-1} \ } \pi_{n-1}(\widetilde{K}^{[n-1]})\\ 
      x \quad & \longmapsto \quad h_{n-1}^{-1} \big( [d_{n}(x)] \big)
\end{align*}

Let $\mathbf{C} $ be the chain complex of $\Lambda$-modules
\begin{equation*}
 \Lambda^{q+1} \cong \big(\Lambda \oplus P\big)^{\ast} \xrightarrow{\ d_{n} \ } C_{n-1}(\widetilde{K}^{[n-1]}) \xrightarrow{ \ d_{n-1} \ } \ \cdots \ \xrightarrow{ \ \quad \ } C_{1}(\widetilde{K}^{[n-1]}) \xrightarrow{ \ \quad \ } \Lambda
\end{equation*}
Then  $Y = ( \mathbf{C}, \varphi', K^{[n-1]})$ is a homotopy system of order $n$. As $C_{i} = 0$ for $i > n$, \ \(
{H}^{n+2}(Y ; \Gamma_{n}Y) = 0$ and, by Proposition 8.3 in \cite{HJB-BB2010}, there is a homotopy system $(\mathbf{C}, 0, X)$ of order $n+1$ realising $Y$, with $X$ an $n$-dimensional $CW$-complex. By construction, $\mathbf{C}(\widetilde{X}) = \mathbf{C}$ and the universal cover of $X$ is $(n-2)$-connected. The inclusion $ i \colon K^{[n-1]} \longrightarrow K$ extends to a map 
\begin{equation*}
    f  \colon X \longrightarrow K = K(G,1)
\end{equation*}
and we may consider $\omega \in H^{1}\!\left( K; \mathbb{Z}/2\mathbb{Z}\right)$ as element of $ H^{1}\!\left( X; \mathbb{Z}/2\mathbb{Z}\right)$.

\begin{prop}\label{prop:realisefree}
 $X$ is a $PD_{n}$-complex with fundamental triple $(G, \omega, \mu)$, that is
 \begin{itemize}
   \item[(i)] $ \mathbb{Z} \cong H_{n}(X; \mathbb{Z}^{\omega}) = \langle [X] \rangle $;
   
   \item[(ii)] $f_{\ast}([X]) = \mu$;
   
   \item[(iii)] $\underline{\phantom{A}} \smallfrown [X] \colon H^{r}(X; \, ^{\omega}\!\Lambda) \longrightarrow H_{n-r}(X ; \Lambda)$ is an isomorphism for every $r \in \mathbb{Z}$.
 \end{itemize}
\end{prop} 

\begin{proof}

\textbf{(i)} \ As $\mathbf{C}(\widetilde{X}) = \mathbf{C}$ is a chain complex of finitely generated free $\Lambda$-modules, the natural map
 \begin{align*}
     \eta_{\mathbf{C}} \colon \mathbb{Z}^{\omega} \otimes_{\Lambda} \mathbf{C} \longrightarrow \Hom_{\Lambda}\big(\,^{\omega}\!\Hom_{\Lambda}(\mathbf{C}, \Lambda), \mathbb{Z}\big)
 \end{align*}
 is an isomorphism. Hence, \( H_{n} \big(X ; \mathbb{Z}^{\omega} \big) = \ker(1 \otimes d_{n}) \cong \ker (\varphi^{+}) \), for 
 \begin{align*}
   \varphi^{+} \colon \Hom_{\Lambda} \big( 
   \Lambda \oplus \Lambda^{q}
   , \mathbb{Z} \big) &\longrightarrow \Hom_{\Lambda}\big( ^{\omega}\!\Hom_{\Lambda}( C_{n-1}(\widetilde{K}^{[n-1]}), \Lambda), \mathbb{Z}\big),\\ 
   \psi \qquad  &\longmapsto \qquad  \psi \circ \varphi
 \end{align*}

Since both $\pi$ and $j$ are surjective, both $\pi^{+}$ and $j^{+}$ are injective, whence
\begin{equation*}
   \ker(\varphi^{+}) = \ker \left( \left( \begin{bmatrix} i & 0 \\ 0 & id \end{bmatrix} \circ j \circ \pi \right)^{+}\right)
       = \ker \left( \begin{bmatrix} 
       i & 0 \\
       0 & id 
       \end{bmatrix}^{+}\right) 
       = \ker \left( \begin{bmatrix} 
       i^{+} & 0 \\
       0 & id 
       \end{bmatrix}\right]
       \cong \ker(i^{+})
\end{equation*}

But $I$ is generated by elements $1 - g \ \ (g\in G)$ and  $(\psi \circ i)(1-g) = 0$ for $\psi \in \Hom_{\Lambda}( \Lambda, \mathbb{Z})$. Hence, $\ker (\varphi^{+}) \cong \Hom_{\Lambda}( \Lambda, \mathbb{Z}) \cong \mathbb{Z}$, generated by $\aug \circ \, pr_{\Lambda} \colon \Lambda \oplus \Lambda^{q} \longrightarrow \mathbb{Z}$, the projection onto the first factor, followed by the augmentation. 

Let $[X] = [1 \otimes x] \in H_{n}(X;\mathbb{Z}^{\omega})$ be the homology class corresponding to $ \aug \circ \, pr_{\Lambda}$ under the isomorphism  \( H_{n} \big(X ; \mathbb{Z}^{\omega} \big) = \ker(1 \otimes d_{n}) \cong \ker (\varphi^{+}) \cong \Hom_{\Lambda}( \Lambda, \mathbb{Z})\). Then $x \in (\Lambda \oplus \Lambda^{q})^{\ast}$ is projection onto the first factor.
  
    \smallskip
  
  \textbf{(ii)\phantom{i}} By the proof of Lemma \ref{lem:nu}, $\nu_{\mathbf{C}(\widetilde{X}), n}\big([X]\big)$ is represented by
  \begin{equation*}
     F^{n-1} \big(\mathbf{C}(\widetilde{X})\big) \longrightarrow I, \quad [\psi] \longmapsto \overline{\psi(d_{n}(x))}
  \end{equation*}

 Thus, given $\psi \in C_{n-1}(\widetilde{X})^{\ast} = C_{n-1}(\widetilde{K})^{\ast}$,
  \begin{align*}
   \overline{\psi(d_{n}(x))} & = \overline{\psi\big(\,^{\omega}\!\ev^{-1}(x \circ \varphi)\big)} \\
                     &  = \,^{\omega}\!\ev \big(\,^{\omega}\!\ev^{-1}(x \circ \varphi)\big)(\psi) \\
                     & =( x \circ \varphi)(\psi) \\
                     &  = (x \circ \begin{bmatrix} i & 0 \\ 0 & id\end{bmatrix} \circ j \circ \pi)(\psi) \\
                     &  = (i \circ pr_{I} \circ j)([\psi]))\\
                     & = h([\psi])
  \end{align*}
Hence, $\nu_{\mathbf{C}(\widetilde{X}), n}([X])$ is the homotopy class of $h$ so that 
\begin{equation*}
\nu_{\mathbf{C}(\widetilde{K}), n}(\mu) = \nu_{\mathbf{C}(\widetilde{X}), n}([X])
 = \nu_{\mathbf{C}(\widetilde{K}), n}(f_{\ast}([X]))
\end{equation*}
By Lemma 2.5 in \cite{Tu}, $\nu_{\mathbf{C}(\widetilde{K}), n}$ is injective, whence $\mu = f_{\ast}([X])$.

\smallskip
\textbf{(iii)} \ First consider  $1 \le i < n-1$.  Then $  H_{i} (\widetilde{X}; \Lambda) = H_{i}({\widetilde{K}}^{[n-1]}; \Lambda) = 0$.  

By the definition of $\varphi$,
  \[
   H^{n-1}(\widetilde{X};  \, ^{\omega}\!\Lambda^{\omega}) = 0
   \]
Moreover, by hypothesis, 
  \begin{equation*}
  H^{n-i}(\widetilde{X}; \, \Lambda^{\omega}) = H^{n-i}({\widetilde{K}^{[n-1]}}; \,\Lambda^{\omega}) 
  \cong H^{n-i}(G;\Lambda^{\omega}) 
  = 0
  \end{equation*}
for $1 < i < n-1$.  Thus
  \[
    \underline{\phantom{A}} \smallfrown (1 \otimes [X]) \colon H^{n-i}(X; \,^{\omega}\!\Lambda) \longrightarrow H_{i}(X ; \Lambda)
  \]
  is an isomorphism for $1 \le  i < n-1$.
    
  Next consider $i = 0$. As $P$ and $\Lambda \oplus P$ are free, $\mathbf{C}(X)$ is a chain complex of free $\Lambda$-modules. Since the (twisted) evaluation map from a finitely generated free $\Lambda$-module to its double dual is an isomorphism, 
  \begin{align*}
     H^{n}(\widetilde{X}; \, ^{\omega}\!\Lambda) 
            & = \, ^{\omega}\!\Hom_{\Lambda}(C_{n}(X), \, ^{\omega}\!\Lambda)/\im(\varphi^{\ast})^{\ast} \\[1ex]
             & = (\Lambda \oplus P)^{\ast\ast}/\im(\varphi^{\ast})^{\ast} \\[1ex]
           & \cong \Lambda \oplus P/\im(\varphi) \\
           & \cong \Lambda/I \\
           & \cong \mathbb{Z}
  \end{align*}
  
  The class, $[\gamma]$, of the image of $(1,0) \in \Lambda \oplus P$ under the (twisted) evaluation isomorphism generates $H^{n}(X; \, ^{\omega}\!\Lambda)$ and so, by Lemma 4.4 of \cite{BB2010},
  \begin{equation*}
    [\gamma] \smallfrown [X] = [\gamma] \smallfrown [1 \otimes x]
        = [\overline{\gamma(x)}. e_{0}] 
        = [e_{0}]  
  \end{equation*}     
 where $e_{0} \in C_{0}(X)$ is a chain representing the base point. Thus, 
 \[
    \underline{\phantom{A}} \smallfrown [X] \colon H^{n}(\widetilde{X} ; \, ^{\omega}\!\Lambda) \longrightarrow H_{0}(\widetilde{X}; \Lambda)
 \]
 is an isomorphism. 
 
Finally, note that by the above,  $\underline{\phantom{A}} \smallfrown (1 \otimes x)$ yields the chain homotopy equivalence
 \begin{equation*}
   \begin{tikzpicture}[scale=0.8]
      \draw
         node (imn-1) at (-3.5,2) {$\im d^{\ast}_{n-1}$}
         node (im2) at (-3.5,0) {$\im d_{2}$}
         node (cn-1) at (0,2) {$C_{n-1}(X)^{\ast}$}
         node (c1) at (0,0) {$C_{1}(X)$}
         node (cn) at (3.5,2) {$C_{n}(X)^{\ast}$}
         node (c0) at (3.5,0) {$C_{0}(X)$}
      ;
   \draw[>->] (imn-1)--(cn-1);
   \draw[->>] (cn-1)--(cn);
   \draw[>->] (im2)--(c1);
   \draw[->>] (c1)--(c0);
   \draw[->] (imn-1)--(im2);
   \draw[->] (cn-1)--(c1) node[midway,left] {$\underline{\phantom{a}}  \smallfrown (1\otimes x)$};
   \draw[->] (cn)--(c0) node[midway, right] {$\underline{\phantom{a}}  \smallfrown (1\otimes x)$};
   \end{tikzpicture}
 \end{equation*}
 Applying the functor $^{\omega}\!\Hom_{\Lambda}(\underline{\phantom{A}}\, , \Lambda)$, we  obtain the chain homotopy equivalence  $\big(\,\underline{\phantom{A}} \smallfrown (1\otimes x)\big)^{\ast}$, inducing isomorphisms
  \begin{align*}
           \big(\,  \underline{\phantom{A}} \smallfrown [X] \big)^{\ast}\colon    H^{0}(X; ^{\omega}\!\Lambda) &\longrightarrow H_{n}(X ; \Lambda) \\ 
          \big(\,  \underline{\phantom{A}} \smallfrown [X] \big)^{\ast}\colon        H^{1}(X; ^{\omega}\!\Lambda) & \longrightarrow H_{n-1}(X ; \Lambda)
 \end{align*}
By Lemma 2.1 in \cite{BB2010}, $\big(\, \underline{\phantom{A}} \smallfrown (1 \otimes x)\big)^{\ast}$ induces an isomorphism in homology if and only if \ $\underline{\phantom{A}} \smallfrown (1 \otimes x)$ does, whence 
\begin{align*}
  \underline{\phantom{A}} \smallfrown [X] \colon H^{0}(X; \, ^{\omega}\Lambda) &\longrightarrow H_{n}(X; \Lambda)\\ 
  \underline{\phantom{A}} \smallfrown [X] \colon H^{1}(X; \, ^{\omega}\Lambda) &\longrightarrow H_{n-1}(X; \Lambda)
\end{align*}
 are isomorphisms.
\end{proof}

Suppose now that $P$ is projective, but not free.
 
Then there is a finitely generated $\Lambda$-module $Q$ and a natural number $q$ with $P^{\ast} \oplus Q \cong \Lambda^{q}$.  The natural isomorphisms
\begin{equation*}
 ( \Lambda \oplus P)^{\ast} \oplus \Lambda^{\infty}  \cong  \Lambda^{\ast} \oplus P^{\ast} \oplus(Q \oplus P^{\ast} \oplus \cdots)
  \cong \Lambda \oplus (P^{\ast}  \oplus Q \oplus P^{\ast} \oplus Q \cdots) 
  \cong\Lambda^{\infty} 
\end{equation*}  
show that $( \Lambda \oplus P)^{\ast} \oplus \Lambda^{\infty}$ is a free $\Lambda$-module. 
 
 Consider the chain complex $\mathbf{D}$
\begin{equation*}
 \begin{tikzpicture}[scale=0.8]
   \draw
      node (0) at (0,1.5) {$0$}
      node (LPL) at (3,1.5) {$(\Lambda \oplus P)^{\ast} \oplus \Lambda^{\infty}$}
      node (KL) at (9,1.5) {$C_{n-1} \big(\widetilde{K}^{[n-1]}\big) \oplus \Lambda^{\infty}$}
      node (1) at (13.5,1.5) {}
      node (Cn-1) at (3,0) {$ C_{n-2}\big(\widetilde{K}^{[n-1]}\big)$}
      node (Cn-2) at (8.5,0) {$ C_{n-3}\big(\widetilde{K}^{[n-1]}\big)$}
      node (2) at (12,0) {$\cdots$}
   ;
   \draw[->] (0)--(LPL);
   \draw[->] (LPL)--(KL) node[midway, above] {$ \begin{bmatrix} d_{n} & 0 \\ 0 & id \end{bmatrix} $};
   \draw[->] (KL)--(1) node[midway, above] {\phantom{a}{$\begin{bmatrix} d_{n-1} & 0 \end{bmatrix} $}\phantom{a}};
   \draw[->] (Cn-1)--(Cn-2) node[midway,above] {$d_{n-2}$};
   \draw[->] (Cn-2)--(2) ;
 \end{tikzpicture}
\end{equation*}

We  attach infinitely many $n$-balls to ${K}^{[n-1]}$ to obtain a $CW$-complex, $K^{\prime}$, whose cellular chain complex coincides with $\mathbf{D}$ in dimensions below $n$. 
Then  ${\widetilde{K'}}^{[n-1]}$  is $(n-2)$-connected, and the Hurewicz homomomorphisms $h_{q} \colon \pi_{q}\big(\widetilde{K'}^{[n-1]}\big) \longrightarrow H_{q}\big(\widetilde{K'}^{[n-1]}\big)$ are isomorphisms for $ q \le n-1$. Defining the map
\begin{align*}
      \varphi^{\prime} \colon ( \Lambda \oplus P)^{\ast} \oplus \Lambda^{\infty} &\longrightarrow \ker(d_{n-1}) = H_{n-1}(\widetilde{K'}^{[n-1]}) \xrightarrow[]{\ h_{n-1}^{-1} \ } \pi_{n-1}(\widetilde{K'}^{[n-1]})\\ 
      x \quad & \longmapsto \quad h_{n-1}^{-1} \big( [d_{n}(x)] \big)
\end{align*}
we obtain the homotopy system $Y' = ( \mathbf{D}, \varphi^{\prime}, K'^{[n-1]})$ of order $n$. As $D_{i} = 0$ for $i > n$, \ \( \widehat{H}^{n+2}(Y' ; \Gamma_{n}Y') = 0$. By Proposition 8.3 in \cite{HJB-BB2010}, there is then a homotopy system $(\mathbf{C}, 0, X')$ of order $n+1$ realising $Y'$, with $X'$ an $n$-dimensional $CW$-complex.
 
Note that $\mathbf{D}$, the chain complex of $X'$, is chain homotopy equivalent to the chain complex $\mathbf{W}$
 \begin{equation*}
   \begin{tikzpicture}[scale=0.7]
   \draw
   node (0) at (0,1.5) {$\cdots$}
   node (PQ1) at (2.5,1.5) {$P^{\ast} \oplus Q$}
   node (PQ2) at (6.25,1.5) {$P^{\ast} \oplus Q$}
   node (PQ3) at (10,1.5) {$P^{\ast} \oplus Q$}
   node (LPQ) at (14.5,1.5) {$( \Lambda \oplus P)^{\ast} \oplus Q$}
   node (Cn-1) at (19,1.5) {$C_{n-1}\big(\widetilde{K}^{[n-1]}\big) $}
   node (1) at (3,0) {$ $}
   node (Cn-2) at (6.25,0) {$C_{n-2}\big(\widetilde{K}^{[n-1]}\big) $}
   node (Cn-3) at (11,0) {$C_{n-3}\big(\widetilde{K}^{[n-1]}\big) $}
   node (2) at (14.25,0) {$ \cdots$}
   ;
   \draw[->] (0)--(PQ1);
   \draw[->] (PQ1)--(PQ2) node [midway, above] {$ \begin{bmatrix} id & 0 \\ 0 & 0 \end{bmatrix}$};
   \draw[->] (PQ2)--(PQ3) node [midway, above] {$ \begin{bmatrix}  0 & 0 \\ 0 & id \end{bmatrix}$};
  \draw[->] (PQ3)--(LPQ) node[midway, above] {$  \begin{bmatrix}  0 & 0 \\ 0 & 0 \\ 0 & id \end{bmatrix}$};
   \draw[->] (LPQ)--(Cn-1) node [midway, above] {$\begin{bmatrix} d_{n} \\ 0 \end{bmatrix} $};
  \draw[->] (1)--(Cn-2) node [midway, above] {$d_{n-1}$};
  \draw[->] (Cn-2)--(Cn-3) node [midway, above] {$d_{n-2}$};
  \draw[->] (Cn-3)--(2);
   \end{tikzpicture}
 \end{equation*}
 
By  Theorem 2 of \cite{W1966}, there is a $CW$-complex, $X$, with cellular chain complex $\mathbf{W}$, homotopy equivalent to $X^{\prime}$. The proof that $X$ realizes $(G, \omega, \mu)$ is similar to the proof of Proposition \ref{prop:realisefree}.  
 
This completes the proof of Theorem A.

 \section{Decomposition as Connected Sum}\label{section:connsum}
 
Wall constructed a new $PD_{n}$-complex from given ones  using the \emph{connected sum of $PD_{n}$-complexes} (c.f.~\cite{W1967}). This allows  $PD_{n}$-complexes to be decomposed as connected sum of other, simpler $PD_{n}$-complexes.

Take $PD_{n}$-complexes $(X_{k}, \omega_{k}, [X_{k}])$ for $k = 1, 2$. Then we may express $X_{k}$ as the mapping cone
\begin{equation*}
       X_{k} = X_{k}^{\prime}\cup_{f_{k}}e_{k}^{n}
\end{equation*}
for suitable $f_{k} \colon S^{n-1} \longrightarrow X_{k}^{\prime}$. Here, $X_{k}^{\prime}$ is an $(n-1)$-dimensional $CW$-complex when $n > 3$, and when $n=3$, $X_{k}^{\prime}$ is 3-dimensional with $H^{3}(X_{k}^{\prime}; B) = 0$ for all coefficient modules, $B$. For $k = 1, 2$, let  $\iota_{k} \colon X_{k}^{\prime} \longrightarrow X_{1}^{\prime} \vee X_{2}^{\prime}$ be the canonical inclusion of the $k^{\text{th}}$ summand and put
\begin{equation*}
     \widehat{f}_{k} :=\iota_{k} \circ f_{k} \colon S^{n-1} \longrightarrow X_{1}^{\prime} \vee X_{2}^{\prime}
\end{equation*} 
so that $\widehat{f}_{k}$ determines an element of $\pi_{n-1}\big(X_{1}^{\prime} \vee X_{2}^{\prime}\big)$. Let  $f_{1} + f_{2} \colon S^{n-1} \longrightarrow X_{1}^{\prime} \vee X_{2}^{\prime}$ represent the homotopy class $[\widehat{f}_{1}] + [\widehat{f}_{2}]$. Then the connected sum,  $X = X_{1} \# X_{2} $, of $X_{1}$ and $X_{2}$ is  the mapping cone of $f_{1} + f_{2}$,
\begin{equation*}
     X_{1} \# X_{2} := \big(X_{1}^{\prime} \vee X_{2}^{\prime}\big) \cup_{f_{1}+f_{2}} e^{n}
\end{equation*}

It follows from the Seifert-van Kampen Theorem that 
\begin{equation*}
    \pi_{1}(X) = \pi_{1}(X_{1}) \ast \pi_{1}(X_{2})
\end{equation*}
The canonical inclusion $in_{k} \colon \pi_{1}(X_{k}) \longrightarrow \pi_{1}(X)$ induces a (left, resp.~right) $\mathbb{Z}[\pi_{1}(X_{k})]$-module structure on any  (left, resp.~right) $\Lambda = \mathbb{Z}[\pi_{1}(X)]$-module. In particular, $\Lambda$ is a $\pi_{1}(X_{k})$-bimodule. By the universal property of the free product, the group homomorphisms $\omega_{X_{k}} = in_{k}^{\ast}\big(\omega_{X}\big)$ uniquely determine a group homomorphism $\omega_{X} \colon \pi_{1}(X) \longrightarrow \mathbb{Z}/2\mathbb{Z}$.
For $k = 1, 2$, let $L_k$ be the functor $\Lambda\otimes_{\mathbb{Z}[\pi_{1}(X_{k})]} \underline{\phantom{A}}\,$.

Let $B$ be the subcomplex of $\mathbf{C}(\widetilde{X})$ containing the $n$-cell attached by $f_{1} + f_{2}$. Then $\mathbf{B}$ is a Poincar{\'{e}} duality chain complex (\cite{HJB-BB2010} p.2361) and it follows from Theorem 2.3 of \cite{BB2010} that $ L_{k}\big(\mathbf{C}(\widetilde{X}_{k})\big)$ is also a Poincar{\'{e}} duality chain complex.

Let $x$ denote the chain representing the $n$-cell attached by $f_{1} + f_{2}$. Repeated application of Theorem 2.3 of \cite{BB2010} shows that $L_{1}\big(\mathbf{C}(\widetilde{X}_{1})\big) + L_{2}\big(\mathbf{C}(\widetilde{X}_{2})\big)$ is a Poincar{\'{e}} duality chain complex. Hence $(X, \omega_{X}, [1 \otimes x])$ is a Poincar{\'{e}} duality complex. This is the \emph{connected sum of $(X_{1}, \omega_{X_{1}}, [X_{1}])$ and $(X_{2}, \omega_{X_{2}}, [X_{2}])$}, introduced by Wall in \cite{W1967}.

We have seen that a necessary condition for a Poincar{\'{e}} duality complex to be a non-trivial connected sum of Poincar{\'{e}} duality complexes is that its fundamental group be a non-trivial free product of groups. We show that, in our context,  this condition is also sufficient.

\begin{mainb}
A $PD_{n}$-complex with $(n-2)$-connected universal cover decomposes as a non-trivial connected sum if and only if its fundamental group decomposes as a non-trivial free product of groups.
\end{mainb}

\begin{proof}
   It only remains to prove sufficiency.
   
   Let $(X, \omega_{X}, [X])$ be a $PD_{n}$-complex with $(n-2)$-connected universal cover and with $\pi_{1}(X) = G_{1} \ast G_{2}$,  for non-trivial groups $G_{1}, G_{2}$. As $\pi_{1}(X)$ is finitely presentable, so are $G_{1}$ and $G_{2}$. For $j = 1,2$, let   $K_{j} = K(G_{j};1)$ be an Eilenberg-Mac Lane space with finite 2-skeleton. Then $K_{1} \vee K_{2}$ is an Eilenberg-Mac Lane space, $K(G_{1} \ast G_{2};1)$, and
\begin{equation*}
    H_{n} (K; \mathbb{Z}^{\omega}) = H_{n}(K_{1}; \mathbb{Z}^{\omega_{1}}) \oplus H_{n}(K_{n};\mathbb{Z}^{\omega_{2}})
\end{equation*}
where $\omega_{j} \in H^{1}(K_{j}; \mathbb{Z}/2\mathbb{Z}) $ is the restriction of the orientation character $ \omega \in H^{1}(K; \mathbb{Z}/2\mathbb{Z})$. Thus, $\mu_{X} = \mu_{1} + \mu_{2}$, with $\mu_{j} \in H_{n} (K_{j} ; \mathbb{Z}^{\omega_{j}})$ for $j = 1,2$.

By the discussion above, if the $PD_{n}$-complex $X_{j}$, with $(n-2)$-connected universal cover realizes the fundamental triple $(G_{j}, \omega_{j}, \mu_{j})$, then the connected sum of $X_{1}$ and $X_{2}$ realizes the fundamental triple of $X$, whence, by the Classification Theorem in \cite{HJB-BB2010}, $X$ is orientedly homotopy equivalent to $X_{1} \# X_{2}$. Hence, it is sufficient to construct realizations of $(G_{j}, \omega_{j}, \mu_{j}) \ (j=1,2).$ 

Let $L_{j}$ be the functor $\Lambda \otimes_{\mathbb{Z}[\pi_{1}(X_{j})]}\underline{\phantom{A}}\,$, so that for $ i \ge 1$
\begin{equation*}
    C_{i}(\widetilde{K}) = L_{1}\big(C_{i}(\widetilde{K}_{1})\big) \oplus L_{2}\big(C_{i}(\widetilde{K}_{2})\big)
\end{equation*}
It follows that 
\begin{equation*}
   F^{n-1}\big(\mathbf{C}(\widetilde{K})\big) = L_{1}\Big(F^{n-1}\big(\mathbf{C}(\widetilde{K}_{1})\big)\Big) \oplus L_{2}\Big(F^{n-1}\big(\mathbf{C}(\widetilde{K}_{2})\big)\Big)
\end{equation*}
and
\begin{equation*}
  I \big(\pi_{1}(X)\big) = L_{1}\big(I(G_{1})\big) \oplus L_{2} \big(I(G_{2})\big)
\end{equation*}
where the canonical inclusion is given by
\begin{equation*}
    L_{j}\big(I(G_{j})\big) \longrightarrow I(G), \quad \mu \otimes \lambda \longmapsto \mu\lambda
\end{equation*}
for $\mu \in \mathbb{Z}[\pi_{1}(X)]$ and $\lambda \in I(G)$ viewed as an element of $I\big(\pi_{1}(X)\big)$.

Let \( \varphi_{j} \colon F^{n-1} \big(\mathbf{C}(\widetilde{K}_{i})\big) \longrightarrow I(G_{j}) \) be a $\mathbb{Z}[G_{j}]$-morphism representing the class $ \nu_{C(\widetilde{K}_{j}), n}(\mu_{j})$. Then the class, $ \nu_{C(\widetilde{K}, n}(\mu)$ of homotopy equivalences is represented by
\[
  \begin{CD}
  L_{1}\Big(F^{n-1}\big(\mathbf{C}(\widetilde{K}_{1}) \big) \Big) \oplus L_{2}\Big(F^{n-1}\big(\mathbf{C}(\widetilde{K}_{2}) \big) \Big)& \quad = \ F^{n-1}\big(\mathbf{C}(\widetilde{K}) \big) \\
  @VV{L_{1}(\varphi_{1}) \oplus L_{2}(\varphi_{2})}V\\
  L_{1}\big(I(G_{1}) \oplus L_{2} \big(I(G_{2})\big) &  = \  I(G)\phantom{ABC}
  \end{CD}
\] 
and it follows from the proof of the analogous proposition for $n = 3$ on pp.269--270 in \cite{Tu}, that $\varphi_{j}$ is a homotopy equivalence of modules. 
By Theorem A, $(G_{j}, \omega_{j}, \mu_{j})$ is realized by a $PD_{n}$-complex, $X_{j}$, with $(n-2)$-connected universal cover. 
\end{proof}

\section{The Centralizer Condition of Crisp}\label{section:crisp}

In the remainder of this paper we shall consider indecomposable
$\mathrm{PD}_n$-complexes with $(n-2)$-connected universal covers.
The arguments of \cite{Cr} for the case $n=3$ apply equally well in higher dimensions.
When $n$ is odd they imply that the indecomposable PD$_n$-complexes 
of this type are either aspherical or have virtually free fundamental group.
Theorem 17 of \cite{Cr} then leads to strong constraints 
on the possible groups in the latter case, as in \cite{Hi12}.
The consequences are different when $n$ is even.
In particular, there may be no simple characterization of the indecomposables.
However, if the fundamental group is indecomposable, virtually free 
and non-trivial then it seems likely that it either has two ends 
or is $\mathbb{Z}/2\mathbb{Z}$.

Let $X$ be a $\mathrm{PD}_n$-complex with $(n-2)$-connected universal cover $\widetilde{X}$,
and let $\pi=\pi_1(X)$ and $\omega=w_1(X)$.
Let $\pi^+=\mathrm{Ker}(\omega)$ and let $X^+$ be the corresponding orientable covering space.

Since $\pi$ is $FP_2$ it acts without inversions on a ``terminal $\pi$-tree" 
$T$ and so is the fundamental group of a finite graph of groups $(\mathcal{G},\Gamma)$,
where $\Gamma=\pi\backslash{T}$, 
all vertex groups are finite or have one end and all edge groups are finite.
(See Theorem VI.6.3 of \cite{DD}.)


A {\it graph of groups} $(\mathcal{G},\Gamma)$ consists of a graph $\Gamma$ 
with origin and target functions $o$ and $t$ from the set of edges 
$E=E(\Gamma)$ to the set of vertices $V=V(\Gamma)$, 
and a family $\mathcal{G}$ of groups $G_v$ 
for each vertex $v$ and subgroups $G_e\leq{G_{o(e)}}$ for each edge $e$,
with monomorphisms $\phi_e:{G_e}\to{G_{t(e)}}$.
(We shall usually suppress the maps $\phi_e$ from our notation.)
All edges are oriented, but we do not use this,
and in considering paths or circuits in $\Gamma$ 
we shall not require that the edges be compatibly oriented.
The {\it fundamental group} of $(\mathcal{G},\Gamma)$ is the group 
$\pi\mathcal{G}$ 
with presentation
\[\langle G_v,~t_e,~\forall{v}\in{V(\Gamma)},~\forall{e}\in{E(\Gamma)}
\mid~t_egt_e^{-1}=\phi_e(g),~\forall{g}\in{G_e},~\forall{e}\in{E(\Gamma)},~t_f=1,~
\forall{f}\in{E(\Upsilon)}\rangle,\]
where $\Upsilon$ is some maximal tree for $\Gamma$.
The generator $t_e$ is the {\it stable letter\/} associated to the edge $e$.
Different choices of maximal tree give isomorphic groups.
We may assume that $(\mathcal{G},\Gamma)$ is {\it reduced\/}: 
if an edge joins distinct vertices then the edge group is isomorphic to 
a proper subgroup of each of these vertex groups.
The corresponding $\pi$-tree $T$ is incompressible in the terminology of \cite{DD},
and so $T$ and $\mathcal{G}$ are essentially unique, 
by Proposition IV.7.4 of \cite{DD}.
We may also assume that $\pi$ is indecomposable as a proper free product, 
by Theorem B, and so $(\mathcal{G},\Gamma)$ is {\it indecomposable}: 
all edge groups are nontrivial.

Let $V_f$ be the subset of vertices $v$ such that $G_v$ is finite.
Every finite subgroup of $\pi$ fixes a vertex of $T$,
and so is conjugate into $G_v$, for some $v\in{V}$. 
(See Corollary I.4.9 of \cite{DD}.)
Thus vertex subgroups are {\it maximal\/} finite subgroups of $\pi\mathcal{G}$.
Following \cite{Cr}, 
we shall say that a vertex of $T$ is finite or infinite if
its stabilizer in $\pi$ is finite or infinite, respectively.

Note that if $\sigma$ is a subgroup of finite index in $\pi$ then it acts on $T$ with stabilizers finite or one-ended,
and so is the fundamental group of a finite graph of groups 
$(\mathcal{G}_\sigma,\Gamma_\sigma)$
where $\Gamma_\sigma=\sigma\backslash{T}$ projects naturally onto $\Gamma$.
However $(\mathcal{G}_\sigma,\Gamma_\sigma)$ may be neither reduced nor indecomposable.

The graph of groups $(\mathcal{G},\Gamma)$ gives rise to a ``Chiswell" exact sequence of right $\mathbb{Z}[\pi]$-modules:
\begin{equation*}
\begin{CD}
0\to\oplus_{v\in{V_f}}\mathbb{Z}[G_v\backslash\pi]@>\Delta>>
\oplus_{e\in{E}}\mathbb{Z}[G_e\backslash\pi]\to{H^1(\pi;\mathbb{Z}[\pi])}\to0,
\end{CD}
\end{equation*}
in which the image of a coset $G_vg$ of $G_v$ in $\pi$ under $\Delta$ is
\[
\Delta(G_vg)=\Sigma_{o(e)=v}(\Sigma_{G_eh\subset{G_v}}G_ehg)-
\Sigma_{t(e)=v}(\Sigma_{G_eh\subset{G_v}}G_ehg).
\]
The outer sums are over edges $e$ and the inner sums are over cosets of $G_e$ in $G_v$.
(See \cite{Ch}.)

If $C$ is a finite subgroup of $\pi$ then the summands $\mathbb{Z}[G_v\backslash\pi]$ and $\mathbb{Z}[G_e\backslash\pi]$
are themselves direct sums of permutation modules,
when considered as right $\mathbb{Z}[C]$-modules.
Poincar\'e duality and a homological argument by devissage give isomorphisms
\[
H_s(C;H_{n+1}(\widetilde{X};\mathbb{Z}))\cong{H_s(C;  {^\omega{H}}^1(\pi;\mathbb{Z}[\pi]))}\cong{H_{s+n}(C;\mathbb{Z})},
\]
for all $s\geq1$. 
(See Lemma 2.10 of \cite{Hi}.)
The work in \cite{Cr} relates these two aspects of $H^1(\pi;\mathbb{Z}[\pi])$,
when $n=3$.

An edge $e$ is a {\it loop isomorphism} at $v$ if $o(e)=t(e)=v$ 
and the inclusions induce isomorphisms $G_e\cong{G_v}$.
It is an {\it MC-tie} if $o(e)\not=t(e)$ and 
$G_e$ has index 2 in each of $G_{o(e)}$ and $G_{t(e)}$.
(We shall give the motivation for this name later.)
If $G$ is a subgroup of $\pi$ let $C_\pi(G)$ and $N_\pi(G)$ denote the
centralizer and normalizer of $G$ in $\pi$, respectively.
If $x\in\pi$ let $\langle{x}\rangle$ be the cyclic subgroup generated by $x$ and $C_\pi(x)=C_\pi(\langle{x}\rangle)$.

\begin{lemma}
\label{Ninf}
If an edge $e$ is a loop isomorphism or an MC-tie then $N_\pi(G_e)$ is infinite.
If a vertex group $G_v$ is finite then $N_\pi(G_v)$ 
is infinite if and only if there is a loop isomorphism at $v$.
\end{lemma}

\begin{proof}
If $e$ is a loop isomorphism at $v$ then the stable letter $t_e$
associated with the edge normalizes $G_e=G_v$.
If $e$ is an MC-tie with ends $u,v$ 
then $G_e$ is normal in each of $G_u$ and $G_v$.
Hence if $\alpha\in{G_u}\setminus{G_e}$ 
and $\beta\in{G_v}\setminus{G_e}$ then $\alpha\beta$ is an element of infinite order in $N_\pi(G_e)$.

Suppose that $G_v$ is finite and $N_\pi(G_v)$ is infinite.
The fixed point set of the action of $G_v$ on a terminal $\pi$-tree
is a nonempty subtree, which is preserved by $N_\pi(G_v)$.
Since $N_\pi(G_v)$ is infinite this subtree must have a nontrivial edge,
with image $e$ in $\Gamma$ having $v$ as one vertex.
Then $G_e=G_v$, since this edge is fixed by $G_v$,
and so $e$ must be a loop isomorphism at $v$,
since $\mathcal{G}$ is reduced.
\end{proof}

Suppose first that $n$ is odd, 
and that $C$ is a finite cyclic subgroup of $\pi$.
Then $H_{n+1}(C;\mathbb{Z})=0$ and $H_{n+2}(C;\mathbb{Z})\cong{C}$.
The arguments of Theorems 14 and 17 of  \cite{Cr} extend immediately to show that 
(i) if $X$ is orientable  and indecomposable then either $X$ is aspherical or $\pi$ is virtually free;
and (ii) if $g\in\pi$ has prime order $p>1$ and $C_\pi(g)$ is infinite then
$p=2$, $w(g)=-1$ and $C_\pi(g)$ has two ends.
We may then apply the analysis of \cite{Hi12} to further constrain the possibilities.
However, implementing the Realization Theorem may be difficult,
since it involves the module $F^{n-1}(C(\widetilde{K}))=\coker(d^*_{n-2})$.
As there is no algorithm for computing the homology 
of a finitely presentable group in degrees $>1$ \cite{GGS},
there may be no algorithm to provide an explicit matrix for $d_{n-2}$ if $n>3$, in general.
This may not be a problem when $\pi$ is virtually free.
In particular, is $S_3*_{\mathbb{Z}/2\mathbb{Z}}S_3$ the fundamental group of 
a $(2k-1)$-connected $\mathrm{PD}_{2k+1}$-complex for any $k>2$?

When $n$ is even and $C$ is finite cyclic
$H_{n+1}(C;\mathbb{Z})\cong{C}$ and $H_{n+2}(C;\mathbb{Z})=0$.
In this case Lemma 2.10 of \cite{Hi} gives
\[
H_1(C; {^\omega{H}}^1(\pi;\mathbb{Z}[\pi]))\cong
{H_{n+1}(C;\mathbb{Z})}\cong{C}.
\]
Let $T$ be a terminal $\pi$-tree, with $e(T)$ ends and $\infty(T)$ vertices with infinite stabilizers,
and let $\xi(T)=e(T)+\infty(T)-1$.
If $g\in\pi$ has prime order then (since $n$ is even) Remark 13 of \cite{Cr} gives
\[
either~\omega(g)=1~and~\xi(T^{\langle{g}\rangle})=1~or~
\omega(g)=-1~and~\xi(T^{\langle{g}\rangle})=-1.
\]
(The $\pi$-tree $T$ is denoted $X$ in \cite{Cr}).
The argument of Theorem 17 of \cite{Cr} then gives the following:

\begin{theorem}
\label{Cinf}
If $n$ is even, $x\in\pi$ has order $m>1$ and $C_\pi(x)$ is infinite then $C_\pi(x)$ 
is virtually $\mathbb{Z}$ and either $\omega(x)=1$ or $4|m$.
Moreover, $x$ is not conjugate into any infinite vertex group.
\end{theorem}
 
\begin{proof}
If $g\in\pi$ has prime order $p$ and $\omega(g)=-1$ 
then $p=2$ and $\xi(T^{\langle{g}\rangle})=-1$.
Hence $g$ does not fix any end or infinite vertex,
and so $T^{\langle{g}\rangle}$ is a non-empty finite tree 
with all vertices finite.
Since $C_\pi(g)$ leaves invariant $T^{\langle{g}\rangle}$, it is finite.
Hence if $C_\pi(g)$ is infinite then $\omega(g)=+1$ and $\xi(T^{\langle{g}\rangle})=1$.
As in \cite{Cr}, it follows that $g$ fixes a ray $(\varepsilon,\varepsilon')$,
but fixes no infinite vertex, 
and $C_\pi(g)$ is virtually $\mathbb{Z}$.

Suppose now that $x\in\pi$ has finite order $m$ and $C_\pi(x)$ is infinite.
If $m=2k$ then $x^k$ has order 2 and $C_\pi(x)\leq{C_\pi(x^k)}$,
so $C_\pi(x^k)$ is infinite.
Hence $\omega(x^k)=1$, and so either $\omega(x)=1$ or $4|m$.

If $p$ is a prime factor of $m$ 
then $x^\frac{m}p$ has order $p$, 
and so does not fix any infinite vertex of $T$.
Hence the same is true of $x$, and so it  is not conjugate into any infinite vertex group.
\end{proof}

A group $\pi$ is virtually $\mathbb{Z}\Leftrightarrow\pi$ has two ends
$\Leftrightarrow\pi$ has a (maximal) finite normal subgroup $F$ such that $\pi/F\cong\mathbb{Z}$ or
the infinite dihedral group $D_\infty=\mathbb{Z}/2\mathbb{Z}*\mathbb{Z}/2\mathbb{Z}$.
(See pp. 129--130 of \cite{DD}.)

We shall apply Theorem \ref{Cinf} together with the 
\noindent {\bf Normalizer Condition\/} [17, Proposition 5.2.4]:

\smallskip
{\it a proper subgroup of a nilpotent group is properly contained in its normalizer,}

\smallskip
and the following simple lemma.

\begin{lemma}
\label{CinN}
If $C$ is a finite subgroup of $\pi$ then 
$C_\pi(C)$ has finite index in $N_\pi(C)$.
If $A*_CB$ or $A*_C\varphi\leq\pi$ and $N_\pi(C)$ is finite or has two ends
then either $N_A(C)=C$ or $N_B(C)=C$ or $[N_A(C):C]=[N_B(C):C]=2$.
\end{lemma}

\begin{proof}
The first assertion is clear, since $Aut(C)$ is finite.
The second follows from consideration of normal forms in $A*_CB$ or $A*_C\varphi\leq\pi$.
\end{proof}

\begin{lemma}
\label{loopiso}
If $e$ is an edge such that $G_{o(e)}$ and $G_{t(e)}$ 
are finite nilpotent groups then either $e$ is a loop isomorphism or it is an MC-tie.
In particular, if $G_{o(e)}$ or $G_{t(e)}$ has odd order then $e$ must be a loop isomorphism.
\end{lemma}

\begin{proof}
This follows from the normalizer condition, Lemma \ref{CinN} and Theorem \ref{Cinf}.
\end{proof}

When  $\pi$ is virtually free the following lemma complements Theorem \ref{Cinf}.

\begin{lemma}
\label{twoends}
If $n$ is even, $\pi$ is virtually free and $g\in\pi^+$ has prime order $p\geq2$ then $N_\pi(\langle{g}\rangle)$ has two ends.
\end{lemma}

\begin{proof}
Let $F$ be a free normal subgroup of finite index in $\pi$,
and let $\rho$ be the  indecomposable factor of $FC$ containing $C$.
Then $N_\rho(C)=N_{FC}(C)$,
and so has finite index in $N_\pi(C)$, 
since $FC$ has finite index in $\pi$.
Thus we may assume that $\pi=\rho$.
Since $p$ is prime, the nontrivial edge stabilizers for the action of  $C=\langle{g}\rangle$
on the terminal $\pi$-tree $T$ are isomorphic to $C$.
Hence $\Gamma$ has just one vertex and all the edges
are loop isomorphisms, so $\pi$ is a semidirect product $C\rtimes{F(r)}$, with $r\geq0$.
Clearly $C$ is normal in this group.
If $r=0$ then $\pi$ is finite and so $\widetilde{X}\simeq{S^n}$.
But then $|\pi|\chi(X)=\chi(S^n)=2$, and $\pi^+=1$, contrary to hypothesis.
Therefore $r>0$, and so $N_\pi(C)$ is infinite,
since it contains $C\rtimes{F(r)}$.
Hence $N_\pi(C)$  has two ends, 
by Theorem \ref{Cinf} and Lemma \ref{CinN}.
\end{proof}

\section{Other Consequences of the Chiswell Sequence and Poincar\'e Duality}\label{section:otherconsequences}

We shall assume henceforth that $n$ is even.
If $\pi$ is finite then $\widetilde{X}\simeq{S^n}$, and so  $|\pi|\leq2$.
Hence $X\simeq{S^n}$ or ${R}P^n$.
If $\pi$ has one end then $\widetilde{X}$ is contractible,
and so $X$ is aspherical.
Hence we may also assume that $\pi$ has more than one end.

While our main concerns  shall be with the case when all vertex groups are finite,
elementary considerations give some complementary results.

\begin{lemma}
\label{HChis}
Let $\pi=\pi\mathcal{G}$ be as above, and let $g\in\pi$ have order $q$ and $\omega(g)=1$.
Then there is an exact sequence
\[
0\to\oplus_{v\in{V_f}}H_1(\langle{g}\rangle;{^\omega(}\mathbb{Z}[G_v\backslash\pi]))\to
\oplus_{e\in{E}}H_1(\langle{g}\rangle; {^\omega(}\mathbb{Z}[G_e\backslash\pi]))
\to{\mathbb{Z}/q\mathbb{Z}}\to0.
\]
\end{lemma}

\begin{proof}
Since $\mathbb{Z}[G_v\backslash\pi]$ is a permutation $\mathbb{Z}[\langle{g}\rangle]$-module and $\omega(g)=1$,
 $H_0(\langle{g}\rangle; {^\omega(}\mathbb{Z}[G_v\backslash\pi]))$ is a free abelian group, for all $v\in{V_f}$.
Since $H_1(\langle{g}\rangle; {^\omega{H}}^1(\pi;\mathbb{Z}[\pi]))\cong\langle{g}\rangle\cong{\mathbb{Z}/q\mathbb{Z}}$ 
and $H_2(\langle{g}\rangle; {^\omega{H}}^1(\pi;\mathbb{Z}[\pi]))=0$, by Lemma 2.10 of \cite{Hi}, 
the result follows from the long exact sequence of homology for $\langle{g}\rangle$ 
associated to the short exact sequence of left $\mathbb{Z}[\pi]$-modules obtained by
conjugating the Chiswell sequence. 
\end{proof}

The result holds also if $\omega(g)=-1$ and no odd power of $g$ is conjugate into a finite vertex group.
Otherwise the righthand term of the short exact sequence may be $\mathbb{Z}/q'\mathbb{Z}$,
where $q'=q$ or $\frac12q$.
However we shall not need to consider the orientation-reversing case more closely.

\begin{theorem}
\label{edgegp}
Let $\pi=\pi\mathcal{G}$ be as above, and let $g\in\pi$ have order $q>1$. Then
\begin{enumerate}
\item 
If  $q=p^r$, for some prime $p$ and $r\geq1$,
and $\omega(g)=1$ then $g$ is conjugate into an edge group.
If $g$ is in a finite vertex group ${G_v}$ then $g$ is conjugate into $G_e$ for some edge $e$ with $v\in\{o(e),t(e)\}$.

\item
Let $g\in{G_e}$, where $e$ is an edge such that $G_{o(e)}$ and $G_{t(e)}$ 
each have one end, and suppose that $\omega(g)=1$.
If $xgx^{-1}\in{G_{e'}}$ for some $x\in\pi$ and edge $e'$ such that $G_{o(e')}$ and $G_{t(e')}$ 
each have one end then $x\in{G_e}$.
Hence $N_\pi(G_e)=G_e$.

\item
If $G_v$ has one end for all $v\in{V}$ and $g\in\pi$ has finite order then $\omega(g)=1$.
\end{enumerate}
\end{theorem}

\begin{proof}
If $g$ has order $p^r$, for some prime $p$ and $\omega(g)=1$ then 
$H_1(\langle{g}\rangle; {^\omega(}\mathbb{Z}[G_e\backslash\pi]))\cong{\mathbb{Z}/p^r\mathbb{Z}}$ 
for at least one edge $e$, by Lemma \ref{HChis},
for otherwise $\oplus_{e\in{E}}H_1(\langle{g}\rangle; {^\omega(}\mathbb{Z}[G_e\backslash\pi]))$ 
has exponent dividing $p^{r-1}$.
Therefore $g$ must fix some coset $G_ex$,
and so $xgx^{-1}\leq{G_e}$.
If $g\in{G_v}$ but is not conjugate into $G_e$ for any edge $e$ with $v\in\{o(e),t(e)\}$ then
the map from $H_1(\langle{g}\rangle; {^\omega(}\mathbb{Z}[G_v\backslash\pi]))$ to
$\oplus_{e\in{E}}H_1(\langle{g}\rangle;{^\omega(}\mathbb{Z}[G_e\backslash\pi]))$
has nontrivial kernel.

If $xgx^{-1}\in{G_e'}$ for some $x\not\in{G_e}$ and 
edge $e'$ with both adjacent vertex groups having one end 
then $H^1(\pi;\mathbb{Z}[\pi])$ has more than one copy 
of the augmentation $\mathbb{Z}[\langle{g}\rangle]$-module $\mathbb{Z}$ 
as a direct summand.
But then $H_1(\langle{g}\rangle;{^\omega{H}}^1(\pi;\mathbb{Z}[\pi]))$ 
would have at least two copies of $\mathbb{Z}/q\mathbb{Z}$ as direct summands,
which would contradict Lemma 2.10 of \cite{Hi}.

If all vertex groups have one end the Chiswell sequence 
reduces to an isomorphism
$H^1(\pi;\mathbb{Z}[\pi])\cong\oplus_{e\in{E}}\mathbb{Z}[G_e\backslash\pi]$.
If $g$ has order $2k$ and $\omega(g)=-1$ then $\omega(xgx^{-1})=-1$ for all $x\in\pi$,
and so $H_1(\langle{g}\rangle;{^\omega(}\mathbb{Z}[G_e\backslash\pi]))$
has exponent dividing $k$, for all edges $e\in{E}$.
Hence $H_1(\langle{g}\rangle; {^\omega{H}}^1(\pi;\mathbb{Z}[\pi]))$
has exponent dividing $k$.
This contradicts Lemma 2.10 of \cite{Hi}.
\end{proof}

There are easy counterexamples to part (1)
if $n$ is odd or if $\omega(g)=-1$.
Using Lemma \ref{twoends}, it can be shown that (1) holds if $q$ is odd.
(We do not need to know this below.)
However, it does not always hold when $q$ is even.
The simplest counter-example is  given by the fundamental group of the double of the nontrivial $I$-bundle 
over the lens space $L(6,1)$, which is an amalgam of two copies of $\mathbb{Z}/6\mathbb{Z}$
over $\mathbb{Z}/3\mathbb{Z}$.

If $\pi\cong{N}\rtimes\mathbb{Z}/p\mathbb{Z}$, 
where $N$ is torsion free and $p$ is an odd prime, 
then all edge groups are ${\mathbb{Z}/p\mathbb{Z}}$.
Since $(\mathcal{G},\Gamma)$ is reduced and indecomposable, 
either all vertex groups have one end or $\Gamma$ has just one vertex 
and $\pi\cong\mathbb{Z}\oplus{\mathbb{Z}/p\mathbb{Z}}$ (by Lemma \ref{twoends}).

\noindent{\bf Example.}  
Let $n\geq4$ be even, and let $M$ be an orientable $n$-manifold 
such that $\widetilde{M}$ is $(n-2)$-connected.
Suppose that $M$ has a self-homeomorphism $g$ of prime order $p$ 
and with nonempty finite fixed point set.
Then $g$ is orientation-preserving, since $n$ is even.
Let $s$ be the number of fixed points, and let $U=M\setminus{N}$, 
where $N$ is a $\langle{g}\rangle$-invariant
regular neighbourhood of the fixed point set.
Then $\mu=\pi_1(U)\cong\pi_1(M)$, since $n>2$.
Let $V=U/\langle{g}\rangle$ and $X=D(V)=V\cup_{\partial{V}}V$.
Then $\widetilde{X}$ is $(n-2)$-connected, by a Mayer-Vietoris argument,
and $\pi=\pi_1(X)\cong(\mu*\mu*F(s-1))\rtimes{\mathbb{Z}/p\mathbb{Z}}$.
If $\mu$ has one end then $\pi_1(X)\cong\pi\mathcal{G}$,
where $(\mathcal{G},\Gamma)$ is a graph of groups,
with $\Gamma$ having two vertices and $s$ edges, both vertex groups $\mu\rtimes\mathbb{Z}/p\mathbb{Z}$
and all edge groups $\mathbb{Z}/p\mathbb{Z}$.

This construction can be generalized,
by starting with a finite group $G$ which acts semifreely and with finite fixed point set on 
one or several $(n-2)$-connected closed $n$-manifolds.
After deleting regular neighbourhoods of the fixed points,
we may hope to assemble the pieces along pairs of boundary components 
with equivalent $G$-actions.
Note that $G$ must have cohomological period dividing $n$,
since it acts freely on the boundary spheres.
The analogous construction in the odd-dimensional cases 
gives only non-orientable examples (with $G$ of order 2).

Explicitly: the 4-dimensional torus $T^4=\mathbb{R}^4/\mathbb{Z}^4$ 
has such self-maps, of orders 2, 3, 4, 5, 6 and 8.
(It also has a semifree action of $Q(8)$ with finite fixed point set.) 
The group $\mathbb{Z}/k\mathbb{Z}$ acts semifreely, with two fixed points,
on $T_k$, the closed orientable surface of genus $k$.
The corresponding diagonal action on $T_k\times{T_k}$ is semifree, with four fixed points.
Similarly, $S^1\times{S^3}$ has an orientation-preserving involution with four fixed points.
(In the latter case doubling the complement of the fixed point set gives a virtually free group 
which is the free product of three two-ended factors.)

If $p$ is prime, a locally smoothable $\mathbb{Z}/p\mathbb{Z}$-action 
on a closed manifold which is orientable over $\mathbb{F}_p$
cannot have exactly one fixed point.
(See Corollary IV.2.3 of \cite{Bre}.)
Thus the above construction always leads to groups of the form $\pi\cong{G}\rtimes{\mathbb{Z}/p\mathbb{Z}}$, 
where $G$ has a nontrivial free factor.
Is there an example with $\pi$ indecomposable and virtually a free product of $\mathrm{PD}_n$-groups?

If $\pi$ is virtually torsion free must the edge groups have cohomological period dividing $n$?
In general, must $\pi$ be virtually torsion-free?
We suspect no, but have no counter-examples.

\section{Virtually Free Fundamental Group}
\label{section:virtuallyfree}

We shall now restrict further to the class of (infinite) virtually free groups.
The known indecomposable examples among manifolds with such groups
are mapping tori of self-homeomorphisms of $(n-1)$-dimensional spherical space forms
and unions of mapping cylinders of double coverings of two such space forms
(twisted $I$-bundles) with homeomorphic boundary.
The fundamental groups have two ends, 
and graph of group structures with just one edge, 
which is a loop isomorphism or an MC-tie, respectively.
(The examples involving mapping cylinders suggested the latter term.)
There are similar constructions involving PD$_{n-1}$-complexes with universal cover
$\simeq{S^{n-1}}$.

Our goal is to show that these examples are essentially all, 
provided that $\pi$ has no dihedral subgroup of order $>2$.
There are examples with dihedral subgroups and two ends,
and there may still be indecomposable examples with infinitely many ends.
(See Section \ref{section:construct} below.)

\begin{lemma}
\label{noisolates}
Let $H$ be a nontrivial subgroup of $G_v\cap\pi^+$.
Then there is an edge $e$ with $v$ as a vertex and such that $G_e\cap{H}\not=1$.
\end{lemma}

\begin{proof}
Let $F$ be a free normal subgroup of finite index in $\pi^+$.
Then $FH$ is the fundamental group of a finite orientable cover of $X$.
If $G_e\cap{H}=1$, for all edges $e$ with $v$ as a vertex, 
then the induced graph of groups structure for $FH$ 
has a vertex group $H$ with all adjacent edge groups trivial,
and so $H$ is a free factor of $FH$.
Therefore $H$ is the fundamental group of an orientable 
$\mathrm{PD}_n$-complex with $(n-2)$-connected universal cover, 
by Theorem B.
But this is impossible, since $n$ is even and $H\not=1$.
\end{proof}

\begin{theorem}
\label{oddorder}
If $\pi$ is virtually free then subgroups of $\pi$ of odd order are metacyclic.
\end{theorem}

\begin{proof}
Let $F$ be a free normal subgroup of finite index in $\pi$
and let $p:\pi\to\pi/F$ be the natural epimorphism.
If $S$ is a finite subgroup of $\pi$ then $FS=p^{-1}p(S)\cong{F}\rtimes{S}$.
On replacing $FS$ by an indecomposable factor, if necessary, we may assume that  $FS\cong\pi\mathcal{G}_S$
where $(\mathcal{G}_S,\Gamma_S)$ is an indecomposable reduced finite graph of groups, 
with vertex groups isomorphic to subgroups of $S$, and with at least one edge, since $|S|>2$.

Suppose first that $S$ is nilpotent, of odd order. 
Then $\Gamma_S$ has just one vertex $v$ and one edge,
which is a loop isomorphism, by Lemma \ref{loopiso}.
Hence $\pi\cong{S}\rtimes\mathbb{Z}$ and so has two ends.
Therefore $\widetilde{X}\simeq{S^{n-1}}$ and so $S$ has periodic cohomology.
Since $S$ is nilpotent of odd order it is cyclic.

In general, $S$ is metacyclic, by Proposition 10.1.10 of \cite{Ro},
since all its Sylow subgroups are cyclic.
\end{proof}

This does not extend to subgroups of even order, as it stands.
However for groups of odd order ``metacyclic" is equivalent to
``having periodic cohomology", 
and in all known examples the vertex groups have the latter property.

\begin{cor}
\label{oddordercor}
If $\pi$ has no subgroup isomorphic to $(\mathbb{Z}/2\mathbb{Z})^2$ 
then all finite subgroups of $\pi$ have periodic cohomology.
\end{cor}

\begin{proof}
The exclusion of $(\mathbb{Z}/2\mathbb{Z})^2$ implies that 
finite 2-groups in $\pi$ are cyclic or quaternionic.
(See Proposition 5.3.6 of \cite{Ro}.)
Since all finite $p$-groups of odd order in $\pi$ are cyclic, 
by Theorem \ref{oddorder},
it follows that all finite subgroups have periodic cohomology.
(See Proposition VI.9.3 of \cite{Bro}.)
\end{proof}

Finite groups with periodic cohomology fall into six families:
\renewcommand{\theenumi}{\Roman{enumi}}
\begin{enumerate}
\item $\mathbb{Z}/m\mathbb{Z}\rtimes{\mathbb{Z}/q\mathbb{Z}}$;

\item $\mathbb{Z}/m\mathbb{Z}\rtimes(\mathbb{Z}/q\mathbb{Z}\times{Q(2^i)})$, $i\geq3$;

\item $\mathbb{Z}/m\mathbb{Z}\rtimes(\mathbb{Z}/q\mathbb{Z}\times{T^*_k})$, $k\geq1$;

\item $\mathbb{Z}/m\mathbb{Z}\rtimes(\mathbb{Z}/q\mathbb{Z}\times{O^*_k})$, $k\geq1$;

\item $(\mathbb{Z}/m\mathbb{Z}\rtimes{\mathbb{Z}/q\mathbb{Z}})\times{SL(2,p)}$, $p\geq5$ prime;

\item $\mathbb{Z}/m\mathbb{Z}\rtimes(\mathbb{Z}/q\mathbb{Z}\times{TL(2,p)})$, $p\geq5$ prime.
\end{enumerate}
Here $m$ is odd, and $m$, $q$ and the order of the quotient 
by the metacyclic subgroup $\mathbb{Z}/m\mathbb{Z}\rtimes{\mathbb{Z}/q\mathbb{Z}}$ are relatively prime.
The first family includes cyclic groups,
dihedral groups $D_{2m}=\mathbb{Z}/m\mathbb{Z}\rtimes_{-1}\mathbb{Z}/2\mathbb{Z}$ 
with $m$ odd, 
and the groups of odd order with periodic cohomology.
The group $Q(2^i)$ is the quaternionic group of order $2^i$, 
with presentation 
\[
\langle{x,y}|x^{2^{i-1}}=1,~x^{2^{i-2}}=y^2,~yxy^{-1}=x^{-1}\rangle,
\]
and $T^*_k$ and $O^*_k$ are the generalized binary tetrahedral 
and octahedral groups, respectively.
Then $T^*_k\cong{O^*_k}'\cong{Q(8)\rtimes{\mathbb{Z}/3^k\mathbb{Z}}}$ 
and has index 2 in $O^*_k$.
If $p$ is an odd prime then $TL(2,p)$ may be defined as follows.
Choose a nonsquare $\rho\in\mathbb{F}_p^\times$, 
and let $TL(2,p)\subset{GL(2,p)}$
be the subset of matrices with determinant 1 or $\rho$.
The multiplication $\star$ is given by $A\star{B}=AB$ 
if $A$ or $B$ has determinant 1, 
and $A\star{B}=\rho^{-1}AB$ otherwise.
Then $SL(2,p)=TL(2,p)'$ and has index 2.
(Note also that $SL(2,3)\cong{T^*_1}$ and $TL(2,3)\cong{O^*_1}$.)

We shall not specify the actions in the semidirect products here,
as these play no role in our arguments. 
We shall only need the following simple facts about such groups, 
which may easily be checked by inspecting the terms of the above list.
Let $G$ be a finite group of even order with periodic cohomology.
If $G$ is not metacyclic then its Sylow 2-subgroup is quaternionic, 
and $G$ has an unique element of order 2, which is central.
Hence if $G$ has a dihedral subgroup $D_{2\ell}$ then it is metacyclic.
Moreover,  $D_{2\ell}'$ is then normal in $G$.
If $G$ is metacyclic or of type IV or VI it has an unique subgroup of index 2,
while if it is of type III or V there is no subgroup.
Groups of type II have three such subgroups.
(See  \cite{W2013} for more on these groups.)

\section{Virtually Free Groups without Dihedral Subgroups}
\label{section:nodihedralsub}

Our strategy for proving Theorem \ref{nodihedral} (the main part of Theorem C)
is to use Theorem \ref{Cinf}, the Normalizer Condition, 
and Lemma \ref{noisolates} to show that the graph has just one edge,
which is either a loop isomorphism or an MC-tie.
Lemma \ref{isoneloop} implies that if $G_v$ is a vertex group of maximal order then
there is either a loop isomorphism or an MC-tie  with $v$ as one vertex.
Lemmas \ref{loop2} and \ref{notlooporMC}  show that if $\pi$ has 
no dihedral subgroup then all edge groups have index $\leq 2$ in adjacent vertex groups.
The main result then follows fairly easily.
We shall assume that $X$ is a $\mathrm{PD}_n$-complex
and $\pi\cong\pi\mathcal{G}$,
where $(\mathcal{G},\Gamma)$ is a reduced, indecomposable finite graph of finite groups,
with at least one edge, and that $\pi$ does not have $D_4$ as a subgroup.
We shall not state these conditions explicitly in the lemmas.

\begin{lemma}
\label{isoneloop} 
At each vertex $v$ there is either a loop isomorphism or an edge $e$ 
with distinct vertices $v,w$ and such that $[G_v:G_e]=2$ and 
$N_\pi(G_e)$ has two ends.
\end{lemma}

\begin{proof}
Let $F$ be a free normal subgroup of finite index in $\pi$.
After replacing $\pi$ by an indecomposable factor of $FG_v$, 
if necessary, we may assume that $\pi=FG_v$.
Since $\pi$ is infinite and indecomposable there is
at least one edge with $v$ as a vertex.
If $G_v$ has prime order then each such edge must be a loop isomorphism.
Thus we may assume henceforth that $|G_v|$ is not prime. 

If $G_v$ is metacyclic but $|G_v|$ is not a power of 2 then
$G_v$ has a cyclic normal subgroup $S$ of odd prime order $p$.
If $G_v$ has order $2^k\geq4$ 
or if $G$ is not metacyclic then it has a central element $g$
of order 2 such that $\omega(g)=1$, 
and we let $S=\langle{g}\rangle$.
In each case $S<\pi^+$.

By Lemma \ref{noisolates}, there is an edge $e$  
with $v$ as one vertex and such that $S\leq{G_e}$.
If both vertices are $v$ then $S$ is normalized by $G_v$ 
and by $t_e$, the stable letter associated to $e$,
since $S$ is the unique subgroup of $G_v$ of order $p$.
The subgroup $\langle{G_v,t_e}\rangle$ has 
infinitely many ends unless $G_e=G_v$.
Hence $G_e=G_v$, and so $e$ is a loop isomorphism, by Theorem \ref{Cinf}.

If $e$ has distinct vertices $v\not=w$ then $G_w$ 
is isomorphic to its image in
$FG_v/F\cong{G_v}$.
Hence $S$ is also normal in ${G}_w$, and $|G_w|\leq|G_v|$.
Therefore $S$ is normal in $G_v*_{G_e}G_w$.
Theorem \ref{Cinf} and Lemma \ref{CinN} together imply that
the normalizer of any finite subgroup of $\pi$ is finite or has two ends.
Hence $[G_v:G_e]\leq2$ and $[G_w:G_e]\leq2$. 
Since $e$ is not a loop isomorphism, $[G_v:G_e]=[G_w:G_e]=2$.
Hence $G_e$ is normal in $G_v*_{G_e}G_w$,
and so $N_\pi(G_e)$ has two ends.
\end{proof}

If $G_v$ has no subgroup of index $2$ 
(e.g., if it is metacyclic of odd order or is of type III or V) 
then Lemma \ref{isoneloop} ensures that  there is a loop isomorphism at $v$.
If $G_v$ has maximal order among finite subgroups of $\pi$ and $[G_v:G_e]=2$ then $e$ is an MC-tie.
The argument for Lemma \ref{isoneloop} shows that if $e$ is not a loop isomorphism 
then it an MC-tie for the induced graph of groups structure for $FG_v$.
However, it is not otherwise  obvious that it must be an MC-tie for the original graph of groups 
$(\mathcal{G},\Gamma)$.

\begin{lemma}
\label{loop2}
Let  $f$ be an edge with both vertices $v$.
If $f$ is not a loop isomorphism then $|G_f|=2$ and $G_v$ is dihedral.
\end{lemma}

\begin{proof}
Suppose that $|G_f|>2$.
Let $g$ be an element of $G_f$ of prime order $p$.
Since the Sylow subgroups of $G_v$ are cyclic or quaternionic,
each Sylow $p$-subgroup has an unique subgroup $S$ of order $p$.
Therefore, if  $t_f$ is the stable letter associated to $f$ then
there is an $a\in{G_v}$ such that $at_fgt_f^{-1}a^{-1}=g^s$ for some $0<s<p$.
Hence $(at_f)^{p-1}$ centralizes $S$.
By Lemma \ref{isoneloop}  there is another edge $e$ which is either a loop isomorphism at $v$ or 
has distinct vertices $u$ and $v$, and such that $[G_v:G_e]=2$ and $N_\pi(G_e)$ has two ends.

If $e$ is a loop isomorphism then $S$ is also centralized by some power of $t_e$.
If $[G_v:G_e]=2$ and $N_\pi(G_e)$ has two ends we  may assume that $S\leq{G_e}$, since $|G_f|>2$,
and then $S$ is centralized by an element of infinite order in $N_\pi(G_e)$.
In each case, we find that $S$ is centralized by a nonabelian free subgroup,
contradicting Theorem \ref{Cinf}.

Therefore we must have $G_f\cong\mathbb{Z}/2\mathbb{Z}$.
If $G_v$ is not dihedral then $G_f$ is central in $G_v$.
But the subgroup generated by $G_v$ and $t_f$ contains a nonabelian free group,
and so we again contradict Theorem \ref{Cinf}.
Since $G_v\not\cong{D_4}$, it has order at least 6.
\end{proof}

This lemma indicates why we could require an MC-tie to have distinct vertices.

\begin{lemma}
\label{1loopiso}
Let $e$ and $f$ be distinct edges with vertices $u,v$ and $v,w$, respectively. 
If $e$ is a loop isomorphism at $v$ or an MC-tie then $w\not=u$ or $v$, 
and $f$ is neither a loop isomorphism nor an MC-tie.
\end{lemma}

\begin{proof} 
Suppose that $e$ is a loop isomorphism at $v$
and that $f$ also has both vertices $v$.
If $f$ is a loop isomorphism then $G_v$ is normalized by the free group generated by the stable letters
$t_e$ and $t_f$, which contradicts Theorem \ref{Cinf}.
Therefore $f$ is not a loop isomorphism, and so Lemma \ref{loop2} applies.

If $e$ is a loop isomorphism and $f$ is an MC-tie with vertices $v,w$ 
then $G_f$ is normalized by $G_v*_{G_f}G_w$ and by some power of $t_e$ 
(since $G_v$ has only a finite number of subgroups of index 2). 
This again leads to a contradiction with Theorem \ref{Cinf}.

Finally, if $u\not=v$ and $e$ is an MC-tie then similar arguments 
show that $w\not=u$ or $v$, and that $f$ is not an MC-tie.
\end{proof}

In particular,
if every edge with $v$ as one vertex is either a loop isomorphism or an MC-tie
then there is just one edge, and so $\pi$ has two ends.

\begin{lemma}
\label{notlooporMC}
Let $f$ be an edge with vertices $v\not=w$.
If  $[G_w:G_f]>2$ then $G_f$ has order $2$, and $G_v$ or $G_w$ is dihedral.
\end{lemma}

\begin{proof}
In order to show that $G_f$ has order 2, 
we may assume without loss of generality that $\pi=FG_w$,
where $F$ is a free normal subgroup of finite index.
Then every finite subgroup of $\pi$ is isomorphic to a subgroup of $G_w$,
and so $G_w$ has maximal order among such subgroups.
We may also assume that $o(f)=v $ and $t(f)=w$.
Clearly $f$ is neither a loop isomorphism nor an MC-tie.

There are edges $e$ and $g$ with vertices $u,v$ and $w,x$, respectively, 
which are loop isomorphisms 
or for which $[G_v:G_e]=2$ or $[G_w:G_g]=2$,
by Lemma \ref{isoneloop}.
In the latter case $g$ must be an MC-tie,  by the maximality of $|G_w|$.
Hence $v\not=w$ or $x$, by Lemma \ref{1loopiso}, and so $g\not=f$.
The subgroups $G_e$ and $G_g$ are centralized by elements of infinite order. 
Hence $G_f$ has a subgroup $H$ which is the intersection 
of two subgroups of index $\leq2$ in $G_f$, 
and which is centralized by these elements.
We shall show that we may assume that
they generate a nonabelian free subgroup of $C_\pi(H)$.

If $e$ and $g$ are each loop isomorphisms then $H=G_f$  
is centralized by powers of the stable letters $t_e$ and $t_f$.
If  $e$ is a loop isomorphism and $g$ is an MC-tie then
$H=G_f\cap{G_g}$ is centralized by powers of $t_e$ and $\alpha'\beta'$,
where $\alpha'\in{G_w}\setminus{G_g}$ and $\beta'\in{G_x}\setminus{G_g}$ do not involve $t_e$.

If $e$ is not a loop isomorphism  then $u\not=v$, 
by Lemma \ref{loop2}, and $G_e$ is normalized by some
$\alpha\beta$, where $\alpha\in{G_u\setminus{G_e}}$ and
$\beta\in{G_u\setminus{G_e}}$.
Suppose that $u\not=w$ or $x$.
If $g$ is a loop isomorphism then $H=G_f\cap{G_e}$ is normalized by powers of $\alpha\beta$ and $t_g$,
If $g$ is an MC-tie then $H=G_f\cap{G_e}\cap{G_g}$ is normalized by powers of
$\alpha\beta$ and of $\alpha'\beta'$.
A similar argument applies if $u=w$ or $x$.

In each case, these pairs generate a nonabelian free subgroup of $C_\pi(H)$,
which contradicts Theorem \ref{Cinf}, unless  $H=1$.
Since $G_f$ is nontrivial and is not $D_4$, 
we then have $G_f=\mathbb{Z}/2\mathbb{Z}$.

We now return to the general case (i.e, we do not assume that $\pi=FG_w$).
Since $G_f=\mathbb{Z}/2\mathbb{Z}$, 
it is central in the Sylow 2-subgroups of $G_v$ and $G_w$, 
and $C_\pi(G_f)$ has two ends or is finite.
Hence either $G_f$ is its own centralizer in one vertex group,
in which case the vertex group is dihedral, or
these Sylow subgroups both have order 4, 
and no element of odd order in either vertex group commutes with $G_f$.
In the latter case the vertex groups are metacyclic groups of the form 
$\mathbb{Z}/m\mathbb{Z}\rtimes_\theta\mathbb{Z}/4\mathbb{Z}$,
where $m$ is odd and $\theta:\mathbb{Z}/4\mathbb{Z}\to(\mathbb{Z}/m\mathbb{Z})^\times$
is injective.
Such groups have an unique subgroup of index 2,  and so $G_f\leq{G_e}$
and $G_f\leq{G_g}$.
But then the earlier argument applies with 
$H=G_f$ to show that $C_\pi(H)$ has a nonabelian free subgroup,
contradicting Theorem \ref{Cinf}.
Hence  $G_v$ or $G_w$ is dihedral.
\end{proof}

\begin{theorem}
\label{nodihedral}
If $\pi$ is infinite, virtually free and indecomposable, 
and no maximal finite subgroup is dihedral then $\pi$ has two ends, 
and its finite subgroups have cohomological period dividing $n$.
\end{theorem}

\begin{proof}
Let $G_v$ be a vertex group of maximal order.
Suppose first that $G_v$ has odd order.
Then $G_v$ has no subgroup of index 2 and none of order 2,
and so every edge $e$ with $v$ as a vertex must be a loop isomorphism,
by Lemmas \ref{loop2} and \ref{notlooporMC}.

Since no maximal finite subgroup is dihedral,
there are no dihedral vertex groups.
Therefore if $G_v$ has even order $>4$ and $f$ is an edge 
with vertices $v,w$ then $[G_v:G_f]\leq2$, 
by Lemma \ref{loop2} and \ref{notlooporMC}.
Since $|G_v|$ is maximal,  
$f$ is either a loop isomorphism or an MC-tie.

In each of these cases there must be just one edge, by Lemma \ref{1loopiso},
and so $\pi$ has two ends.

Finally, if all vertex groups have order 4 then they are cyclic, 
and all proper edge groups have order 2.
Hence there is an unique subgroup $S$ of order 2.
Clearly $S\leq\pi^+$,
and so $\pi=C_\pi(S)$ has two ends, by Lemma \ref{twoends}.

Since $\pi$ has two ends, $\widetilde{X}\simeq{S^{n-1}}$, 
and so finite subgroups of $\pi$ have cohomological period dividing $n$.
\end{proof}

In  the final case there are three possibilities: $\pi\cong\mathbb{Z}\oplus\mathbb{Z}/4\mathbb{Z}$,
$\mathbb{Z}/4\mathbb{Z}\rtimes_{-1}\mathbb{Z}$ 
or $\mathbb{Z}/4\mathbb{Z}*_{\mathbb{Z}/2\mathbb{Z}}\mathbb{Z}/4\mathbb{Z}$.

\begin{cor}
If $\pi$ has no element of even order then $\pi\cong{S}\rtimes\mathbb{Z}$,
where $S$ is a finite metacyclic group of odd order and of cohomological period dividing $n$. \hfill $\qed$
\end{cor}

In particular, if $n$  is a power of 2 then $S$ must be cyclic. 
(See Exercise VI.9.6 of \cite{Bro}.)

When there is 2-torsion,
$\pi$ need not be an extension of $\mathbb{Z}$ by a  finite normal subgroup.
For example, if $MC$ is the mapping cylinder of the double cover of a lens space 
$L=L(2m,q)$ and $X=D(MC)=MC\cup_LMC$ is the double then $\pi$ 
is an extension of the infinite dihedral group $D_\infty$ by $\mathbb{Z}/m\mathbb{Z}$,
and $\widetilde{X}\cong{S^3}\times\mathbb{R}$.

\begin{theorem}
\label{mtorus}
Let $\pi$ have two ends and maximal finite normal subgroup $F$.
If $\pi$ is the fundamental group of an orientable $\mathrm{PD}_{n}$-complex $X$ with $\pi_1(X)\cong\pi$
and $\widetilde{X}\simeq{S^{n-1}}$ then $\pi\cong{F}\rtimes_\theta\mathbb{Z}$, and $X$ is a mapping torus.
\end{theorem}

\begin{proof}
The group $\pi$ has a maximal subgroup $\sigma$ of index $\leq2$ which contains $F$ and
such that $\sigma/F\cong\mathbb{Z}$.
Since $\widetilde{X}\simeq{S^{n-1}}$, the covering space $X_F\simeq{S^{n-1}/F}$ associated to $F$
is a PD$_{n-1}$-complex, and so the covering space $X_\sigma$ associated to $\sigma$ is the
mapping torus of a self homotopy equivalence of $F$.
Hence $\chi(X_\sigma)=0$, and so $\chi(X)=0$ also.
But if $\sigma\not=\pi$ then $\pi/\pi'$ is finite.
Since $cd_\mathbb{Q}\pi=1$, it follows from the spectral sequence for the universal
covering that $H_q(X;\mathbb{Q})=0$ for $0<q<n-1$.
This is also the case when $q=n-1$, by Poincar\'e duality.
Hence $\chi(X)=2$ (since $n$ is even).
This is a contradiction.
Therefore $\sigma=\pi\cong{F}\rtimes\mathbb{Z}$ and $X$ is a mapping torus.
\end{proof}

Theorems \ref{nodihedral} and \ref{mtorus} together establish Theorem C of the introduction.

It remains an open question whether the conclusion of Theorem C must hold 
if $\pi$ has a dihedral maximal finite subgroup.
(There are such examples with $\pi$ having two ends -- see Theorem \ref{wang} below.)
Lemmas \ref{1loopiso} and \ref{notlooporMC} impose some restrictions, 
but leave open the possibility that, for instance, there might be a $\mathrm{PD}_{2k}$-complex $X$
with $(2k-2)$-connected universal cover and $\pi=\pi_1(X)\cong\pi\mathcal{G}$,
where the underlying graph $\Gamma$ is a cycle of length four, the vertex groups are dihedral,
and the edges are alternately MC-ties or have edge group of order 2.

\section{Construction of Examples}
\label{section:construct}

Every finite group $F$ with cohomological period $m$ is the fundamental group 
of an orientable $\mathrm{PD}_{m-1}$-complex with universal cover $\simeq{S^{km-1}}$, for all $k\geq1$ \cite{Sw,W1967}.
Since $m$ is even, such complexes are odd-dimensional.
(In fact, the only non-orientable quotients of finite group actions on spheres 
are the even-dimensional real projective spaces $RP^{2n}$.) 
We may use such complexes to realize  groups with two ends.

\begin{theorem}
\label{wang}
If $\pi\cong{F}\rtimes_\theta\mathbb{Z}$ then there is a $\mathrm{PD}_{km}$-complex 
$X$ with $\pi_1(X)\cong\pi$ and $\widetilde{X}\simeq{S^{km-1}}$ 
if and only if $F$ has cohomological period dividing $km$ and
$H_{km-1}(\theta;\mathbb{Z})$ is multiplication by $\pm1$.
If $|F|>2$ then $X$ is orientable if and only if
$H_{km-1}(\theta;\mathbb{Z})=1$.
\end{theorem}

\begin{proof}
Let $X_F$ be a based orientable $\mathrm{PD}_{2k-1}$-complex with fundamental group 
$F$ and $\widetilde{X}\simeq{S^{2k-1}}$.
If $H_{2k-1}(\theta;\mathbb{Z})=\pm1$ then there is a self homotopy equivalence 
$f$ of $X_F$ which induces $\theta$ \cite{Pl}.
The mapping torus of $f$ is then a $\mathrm{PD}_{2k}$-complex with fundamental group 
$\pi$ and universal cover $\simeq{S^{2k-1}}$.

Conversely, if  a $\mathrm{PD}_{2k}$-complex $Y$ has fundamental group $\pi$ 
and universal cover $\widetilde{Y}\simeq{S^{mk-1}}$ then $F$ has 
cohomological period dividing $mk$, since it acts freely on $\widetilde{Y}$, 
and the condition $H_{km-1}(\theta;\mathbb{Z})=\pm1$ 
follows from the Wang sequence for the projection of $Y$ 
onto $S^1$ corresponding to the epimorphism $\pi\to\pi/F\cong\mathbb{Z}$.
(See Theorem 11.1 of \cite{Hi} for the case $n=4$.)

If $|F|>2$ then $H_{2k-1}(\theta;\mathbb{Z})=1$
 (as an automorphism of $H_{2k-1}(F;\mathbb{Z})\cong\mathbb{Z}/|F|\mathbb{Z}$)
if and only if ${H_{2k-1}(f;\mathbb{Z})=1}$ (as an automorphism of
 $H_{2k-1}(X;\mathbb{Z})\cong\mathbb{Z}$)
if and only if ${X}$ is orientable.
\end{proof}

Theorem \ref{wang} is essentially Theorem D of the introduction.

In particular, when $n$ is divisible by 4 there are examples 
with $\pi\cong{D_{2m}}\times\mathbb{Z}$, for odd $m>1$.
These do not satisfy the hypotheses of Theorem C.

Suppose now that $\pi\cong{G*_FH}$,
where $G$ and $H$ are finite groups with periodic cohomology and $[G:F]=[H:F]=2$.
Let  $n$ be a multiple of the cohomological periods of $G$ and $H$.
However there is one subtlety: we must be able to choose $\mathrm{PD}_{n-1}$ complexes $X_G$ and $X_H$
with fundamental groups $G$ and $H$ and universal covers $\simeq{S^{n-1}}$ in such a way
that the double covers associated to the subgroups $F$ are homotopy equivalent.
For then we may construct a $\mathrm{PD}_n$-complex $X$ with fundamental group $\pi$ and 
$\widetilde{X}\simeq{S^{n-1}}$ by gluing together two mapping cylinders via a homotopy equivalence
of their ``boundaries".
See Chapter 11 of \cite{Hi} for an example with $n=4$, $G=Q(24)$, 
$H=\mathbb{Z}/3\mathbb{Z}\times{Q(8)}$ and $F=\mathbb{Z}/12\mathbb{Z}$
 where this construction cannot be carried through.
 (The difficulty is that $\mathrm{PD}_3$-complexes with fundamental group $G$ or $H$ have
unique homotopy types: the double covers corresponding to $F$ are lens spaces which are not homotopy equivalent.
Similar examples should exist in higher dimensions.)

Since putting this paper on the arXiv in May 2016 we have learned that Theorem \ref{wang} 
and much of the subsequent discussion may also be found in the final section of \cite{GG}.
The paper \cite{GG} also gives estimates of the number of homotopy types realizing a given fundamental group.
However we have chosen to retain our independent treatment as it is brief and is a natural complement to our
more substantial earlier results.

In general, it is not known when such a $\mathrm{PD}_n$-complex is homotopy equivalent 
to a closed $n$-manifold.
(This question leads to delicate issues of algebraic number theory \cite{HM}.)
We note also that  there has been extensive research on the mixed-spherical space form
problem, on the fundamental groups of manifolds with universal covering space 
$S^n\times\mathbb{R}^k$ for $n, k>0$.
A recurring theme is the role of finite dihedral subgroups.
See \cite{HP} for a survey of recent progress.

\section{Must Finite Subgroups Have Periodic Cohomology?}\label{section:periodic}

There remains the key question of whether the group $(\mathbb{Z}/2\mathbb{Z})^2$ ever arises in this context.
Suppose that $Y$ is a $\mathrm{PD}_{2k}$-complex with $(2k-2)$-connected universal cover
 and virtually free fundamental group,
and that $\pi_1(Y)$ has a subgroup $C\cong(\mathbb{Z}/2\mathbb{Z})^2$.
Let $F$ be a free normal subgroup of finite index in $\pi_1(Y)$.
We may assume that $F<\pi_1(Y)^+$.
Then $Y$ has a finite cover $Y_{FC}$ with fundamental group $FC\cong{F}\rtimes{C}$.
As in Theorem \ref{oddorder}  some indecomposable factor $X$ of $Y_{FC}$
has fundamental group $F(r)\rtimes{C}$,
for some $r>1$.
(Since $\pi$ cannot be finite of order 4 and  
$C$ does not have periodic cohomology, $r\not=0$ or 1.)
We may assume that $\pi\cong\pi\mathcal{G}$, 
where $(\mathcal{G},\Gamma)$ is an indecomposable reduced finite graph of groups, with all vertex groups
$G_v\cong{C}$ and all edge groups $G_e$ of order 2.
In particular, every edge group is orientable, by Theorem \ref{Cinf}.

Since $\pi$ is virtually free it has a well-defined virtual Euler characteristic
\[
\chi^{virt}(\pi)=\chi(F(r))/|C|=\frac{1-r}4.
\]
We also have $\chi^{virt}(\pi)=\frac14|V|-\frac12|E|$, since  $(\mathcal{G},\Gamma)$ 
is indecomposable and reduced.
Moreover, $\chi(X)=2\chi^{virt}(\pi)$, 
by the multiplicativity of (virtual) Euler characteristic for finite covers (passage to finite-index subgroups),
and so $\chi^{virt}(\pi)\in\frac12\mathbb{Z}$.
Hence $|V|$ is even.

Suppose first that $X$ is not orientable. Then $\omega|_C\not=1$,
since $\omega(F)=1$.
Hence if $e$ and $f$ are two edges with $o(e)=o(f)=v$
then $G_e=G_f=\mathrm{Ker}(\omega|{G_v})$.
If $e\not=f$ then $C_\pi(G_e)$ contains a nonabelian free subgroup.
Hence there is at most one edge at each vertex.
Since $\Gamma$ is connected, there is just one edge $e$,
which must have distinct vertices, 
for otherwise $C_\pi(G_e)$ would have a nonabelian free subgroup.
Hence $\pi\cong{G_u*_{G_e}G_v}\cong(\mathbb{Z}/2\mathbb{Z})\times{D_\infty}$
has two ends, and so $\widetilde{X}\simeq{S^{2k-1}}$.
But then $C$ has periodic cohomology, which is false.
Therefore $X$ must be orientable.

Since $X$ is finitely covered by $\#^r(S^3\times{S^1})$, $H_2(X;\mathbb{Q})=0$,
and so $\chi(X)$ is even.
Hence $\chi^{virt}(\pi)$ is integral.
Moreover, if there is a vertex $v$ of valency $\leq2$ then 
there is an epimorphism $f:\pi\to\mathbb{Z}/2\mathbb{Z}$
which is nontrivial on $G_e$, for all edges $e$ with $v$ as one vertex.
But then $\mathrm{Ker}(f)\cong\pi\widetilde{\mathcal{G}}$, 
where  $(\widetilde{\mathcal{G}},\Gamma)$ is a graph of groups with all vertex groups of order 2 and with trivial edge
groups for the edges with $v$ as one vertex.
This is impossible if the double cover is orientable.
If there is a vertex $w$ with valency $>3$ then two edges with $w$ as one vertex
have the same edge group $G_e<G_w$, and $C_\pi(G_e)$  contains a nonabelian free subgroup.
Therefore each vertex of $\Gamma$ has valency 3, so $2|E|=3|V|$.
In summary, 
\[
X~is~orientable, ~vertices~of~\Gamma~have~valence~3,~|V|~is ~even~and~r=1+2|V|\equiv1~mod~(4).
\]

The simplest example meeting these criteria has $V=\{v,w\}$ and $E=\{a,b,c\}$, 
with each edge having origin $v$ and target $w$.
Then
\[
G_v=\langle{a,b,c}\mid{a^2=b^2=c^2=1,~c=ab}\rangle
\]
and
\[
G_w=\langle{a',b'},c'\mid(a')^2=(b')^2=(c')^2=1,~c'=a'b'\rangle.
\]
The edge groups are $G_a=\langle{a}\rangle$, $G_b=\langle{b}\rangle$ and $G_c=\langle{ab}\rangle$,
as subgroups of $G_v$,
and for each edge $x$ the monomorphism $\phi_x:G_x\to{G_w}$ is given by $\phi_x(x)=x'$.
The edge $a$ is a maximal tree in $\Gamma$.
Let $t,u$ be stable letters corresponding to the other edges.
Then $\pi=\pi\mathcal{G}$ has the presentation
\[\langle{G_v,G_w,t,u}\mid{a'=a, ~b'=tbt^{-1},~a'b'=uabu^{-1}}\rangle,\]
which simplifies to
\[
\langle{a,b,t,u}\mid{a^2=b^2=(ab)^2=1},~atbt^{-1}=tbt^{-1}a=uabu^{-1}\rangle.
\]
Is this the fundamental group of an orientable $\mathrm{PD}_{2k}$-complex with $(2k-2)$-connected universal cover?
Can the arguments of \S2 of \cite{Cr} be tweaked to rule this out?


\end{document}